\documentclass[12pt,english]{article}
\usepackage[useregional]{datetime2}
\usepackage[margin=1.2in]{geometry}
\usepackage{graphicx,latexsym}
\usepackage{setspace}
\usepackage{comment}
\usepackage{pdfcomment}
\usepackage{enumitem}
\usepackage{mathtools}
\usepackage{tikz}
\usepackage{framed}
\usepackage{amsfonts, amssymb, amsmath, amsthm, bm, hyperref}
\usepackage[utf8]{inputenc}
\usepackage[normalem]{ulem}
\usepackage{nicematrix}
\usepackage{mathrsfs}

\makeatletter
\DeclareFontFamily{OMX}{MnSymbolE}{}
\DeclareSymbolFont{MnLargeSymbols}{OMX}{MnSymbolE}{m}{n}
\SetSymbolFont{MnLargeSymbols}{bold}{OMX}{MnSymbolE}{b}{n}
\DeclareFontShape{OMX}{MnSymbolE}{m}{n}{
    <-6>  MnSymbolE5
   <6-7>  MnSymbolE6
   <7-8>  MnSymbolE7
   <8-9>  MnSymbolE8
   <9-10> MnSymbolE9
  <10-12> MnSymbolE10
  <12->   MnSymbolE12
}{}
\DeclareFontShape{OMX}{MnSymbolE}{b}{n}{
    <-6>  MnSymbolE-Bold5
   <6-7>  MnSymbolE-Bold6
   <7-8>  MnSymbolE-Bold7
   <8-9>  MnSymbolE-Bold8
   <9-10> MnSymbolE-Bold9
  <10-12> MnSymbolE-Bold10
  <12->   MnSymbolE-Bold12
}{}

\let\llangle\@undefined
\let\rrangle\@undefined
\DeclareMathDelimiter{\llangle}{\mathopen}%
                     {MnLargeSymbols}{'164}{MnLargeSymbols}{'164}
\DeclareMathDelimiter{\rrangle}{\mathclose}%
                     {MnLargeSymbols}{'171}{MnLargeSymbols}{'171}
\makeatother

\hypersetup{
  colorlinks=true,
  allcolors=blue
}
\usepackage[all]{xy}
\usepackage[backend=biber, style=numeric-comp, natbib=false, giveninits=true, url=false, isbn=false, doi=true, eprint=true]{biblatex}

\reversemarginpar

\renewbibmacro{in:}{%
\ifentrytype{article}{}{\printtext{\bibstring{in}\intitlepunct}}}

\newtheorem{theorem}{Theorem}[section]
\newtheorem{proposition}[theorem]{Proposition}

\newtheorem{lemma}[theorem]{Lemma}
\newtheorem{corollary}[theorem]{Corollary}

\theoremstyle{definition}
\newtheorem{definition}[theorem]{Definition}

\theoremstyle{remark}
\newtheorem{remark}[theorem]{Remark}

\newtheorem{example} [theorem]{Example}

\newcommand{\N}{\mathbb{N}}
\newcommand{\Z}{\mathbb{Z}}
\newcommand{\ud}{\mathrm{d}}
\newcommand{\R}{\mathbb{R}}
\def\C{\mathbb{C}}
\newcommand{\bA}{{\mathbf{A}}}
\newcommand{\bB}{{\mathbf{B}}}
\newcommand{\bC}{{\mathbf{C}}}
\newcommand{\bD}{{\mathbf{D}}}
\newcommand{\bF}{{\mathbf{F}}}
\newcommand{\bI}{{\mathbf{I}}}
\newcommand{\bJ}{{\mathbf{J}}}
\newcommand{\bK}{{\mathbf{K}}}
\newcommand{\bL}{{\mathbf{L}}}
\newcommand{\bP}{{\mathbf{P}}}

\newcommand{\bR}{{\mathbf{R}}}
\newcommand{\bS}{{\mathbf{S}}}
\newcommand{\bT}{{\mathbf{T}}}
\newcommand{\bY}{{\mathbf{Y}}}
\newcommand{\bZ}{{\mathbf{Z}}}

\newcommand{\bAK}{{\bA_\bK}}

\newcommand{\fA}{{\bA_c}} 
\newcommand{\fB}{{\bB_c}} 
\newcommand{\fC}{{\bC_c}} 
\newcommand{\fD}{{\bD_c}} 
\newcommand{\fK}{{\bK_c}} 

\newcommand{\cC}{\mathcal{C}}
\newcommand{\cD}{\mathcal{D}}
\newcommand{\cH}{\mathcal{H}}
\newcommand{\cK}{\mathcal{K}}

\newcommand{\cB}{\mathcal{B}}
\newcommand{\e}{\varepsilon}
\newcommand{\coleq}{\coloneqq}
\newcommand{\eqcol}{\eqqcolon}
\newcommand{\Abso}[1]{{\left\lvert#1\right\rvert}}
\newcommand{\abso}[1]{{\lvert#1\rvert}}
\newcommand{\norm}[1]{{\lVert#1\rVert}}
\newcommand{\Norm}[1]{{\left\lVert#1\right\rVert}}
\newcommand{\todo}[1]{{\color{red}\P #1 \P}}

\DeclareMathOperator{\diag}{diag}

\DeclareMathOperator{\Span}{span}

\let\Re\relax

\DeclareMathOperator{\Re}{Re}

\DeclareMathOperator{\im}{im}
\DeclareMathOperator{\meas}{meas}

\DeclareMathOperator{\tridiag}{tridiag}
\DeclareMathOperator{\trace}{Tr}

\newcommand{\bv}{{\mathbf v}}
\newcommand{\bw}{{\mathbf w}}
\newcommand{\Bm}{{\mathbf m}}
\newcommand{\bn}{{\mathbf n}}
\newcommand{\bx}{{\mathbf x}}
\newcommand{\Cnn}{\C^{n\times n}}
\newcommand{\dd}[1][x]{\,\operatorname{d}\!#1}

\newcommand{\ddt}{\frac{\dd[]}{\dd[t]}}
\newcommand{\ip}[2]{\langle {#1}, {#2} \rangle}
\newcommand{\mHC}{{m_{HC}}}

\newcommand{\mC}{{\mathbf{C}}}
\newcommand{\mCH}{\mC_H} 
\newcommand{\mCS}{\mC_S} 

\newcommand{\mI}{{\mathbf{I}}}
\newcommand{\mJ}{{\mathbf{J}}}
\newcommand{\mP}{{\mathbf{P}}}
\newcommand{\mQ}{{\mathbf{Q}}}
\newcommand{\mR}{{\mathbf{R}}}

\newcommand{\mU}{{\mathbf{U}}}
\newcommand{\mV}{{\mathbf{V}}}
\newcommand{\mW}{{\mathbf{W}}}

\newcommand{\tbC}{\widetilde{\bC}}

\newcommand{\tm}{\widetilde{m}}

\newcommand{\makegray}[1]{{\color{gray} #1}}
\newcommand{\makeorange}[1]{{\color{orange} #1}}
\newcommand{\makered}[1]{#1} 
\newcommand{\maketeal}[1]{{\color{teal} #1}}

\newcommand{\bigO}{{\mathcal{O}}}

\newcommand{\sphere}{\mathbb{S}}
\newcommand{\torus}{\mathbb{T}}
\newcommand{\range}{\im}

\begin{document}

\title{Hypocoercivity in Hilbert spaces}
\author{F.~Achleitner\thanks{Vienna University of Technology, Institute of Analysis and Scientific Computing, Wiedner Hauptstr. 8-10, A-1040 Wien, Austria, franz.achleitner@tuwien.ac.at},
A.~Arnold\thanks{Vienna University of Technology, Institute of Analysis and Scientific Computing, Wiedner Hauptstr. 8-10, A-1040 Wien, Austria, anton.arnold@tuwien.ac.at} \thanks{corresponding author},
V.~Mehrmann\thanks{Technische Universit\"at Berlin, Institut f.~Mathematik,  MA 4-5, Stra\ss{}e des 17.~Juni 136, D-10623 Berlin, mehrmann@math.tu-berlin.de}, 
and E.A.~Nigsch\thanks{Vienna University of Technology, Institute of Analysis and Scientific Computing, Wiedner Hauptstr. 8-10, A-1040 Wien, Austria, eduard.nigsch@tuwien.ac.at}
} 

\date{\today}

\maketitle

\begin{abstract}
The concept of hypocoercivity for linear evolution equations with dissipation is discussed and equivalent characterizations that were  developed  for the finite-dimensional case are extended to separable Hilbert spaces. Using the concept of a hypocoercivity index,  quantitative estimates on the short-time and long-time decay behavior of a hypocoercive system are derived. As a useful tool for analyzing  the structural properties, an infinite-dimensional staircase form is also derived and connections to linear systems and control theory are presented. Several examples  illustrate  the new concepts and the  results are applied to the Lorentz kinetic equation.
\end{abstract}


\section{Introduction}

The characterization of the decay behavior of linear dynamical systems is  a classical topic. Recently, this topic has received a boost by the introduction of the concept of hypocoercivity in \cite{Villani} and its subsequent systematic study for systems of linear finite-dimensional ordinary and differential-algebraic equations \cite{AAC2,AAM21,AAM22} as well as for some models of partial differential equations~\makered{\cite{He06,DeSa09,GaMi13,DoMoSc15,AAC16,BDMMS20,ADSW21,AAM22Oseen}}.

In this paper, our aim is to study the short- and long-time behavior of linear evolution equations
\begin{equation}\label{lineveq}
\left\{
 \begin{aligned}
 \dot x(t) &= \bA x(t) = -\bC x(t), \\
 x(0) &= x_0 \in \cH 
\end{aligned}
\right.
\end{equation}
with $x\colon [0, \infty) \to \cH$, where $\bA = -\bC$ is a bounded operator on a separable Hilbert space $\cH$.

\makered{While the long-time analysis and in particular the exponential decay of solutions of dynamical systems is a classical topic, the short-time behavior of dissipative systems has only recently attracted more attention \cite{GaMi13, AAC2}. To understand and modify the short-time (or transient behavior) is a classical topic in control theory, see e.g. \cite{CurtainZwart,IlcRT07,Isi85}. A typical infinite dimensional application are hypocoercive Fokker-Planck equations (with a similar setting like the Lorentz equation studied in \S\ref{sec:lorentz}). These models are e.g. used to model the lay-down process of fibers in the fleece production \cite{Klar2019}. For the quality of this industrial process the convergence and mixing properties on a \emph{finite initial time interval} are crucial. 
}

In the finite-dimensional case, i.e., when $\bA$ is a matrix, a useful and by now quite well-studied notion is its \emph{hypocoercivity index} (see~\cite{AAC, AAC2, AAM22}), which characterizes not only certain algebraic properties of the system \eqref{lineveq} but also can be used to obtain quantitative information about its initial and long-time decay behavior. In this paper, we will extend the study of the hypocoercivity index and its use in the analysis of the long- and short-time behavior to linear operators on infinite-dimensional Hilbert spaces. As is to be expected, several modifications are necessary for the infinite-dimensional setting. Before we can start discussing these, we introduce our basic notation.

Throughout, $\cH$ will denote a Hilbert space, which we will always assume to be separable. We denote  by $\cB(\cH)$ the space of all bounded linear operators on~$\cH$, with the usual understanding that the domain is the full space $\cH$.  
For $\bA \in \cB(\cH)$ we denote by $\bA^*$ its Hilbert space adjoint. A self-adjoint 
operator $\bA$ is called \emph{nonnegative} (\emph{nonpositive}) if $\langle \bA x, x \rangle \ge 0$ ($\langle \bA x, x \rangle \le 0$) for all $x \in \cH$, which is denoted by $\bA\geq0$ ($\bA\leq0$), and \emph{positive} if $\langle \bA x, x \rangle > 0$ for all $\cH \setminus \{0\}$, which is denoted by $\bA>0$. $\bA$ is called \emph{dissipative} if $\Re \langle \bA x, x \rangle \le 0$ for all $x \in \cH$, which is the case if and only if $\langle \bA x, x \rangle + \langle x, \bA x \rangle \le 0$, i.e., if the self-adjoint (Hermitian) part $\bA_H \coleq (\bA + \bA^*)/2$ of $\bA$ is nonpositive. $\bA$ is called \emph{accretive} if $-\bA$ is dissipative. The unique nonnegative square root of a nonnegative self-adjoint operator $\bR$ is denoted by $\sqrt{\bR}$. For the definition of the direct sum of Hilbert spaces we refer to \cite{DS,Conway}.

We consider  operators $\bC$ which decompose as $\bC = \bR - \bJ$, where the operators
\[ 
\bR = \bC_H = \frac{1}{2}(\bC + \bC^*) \textrm{ resp.\ } -\bJ = \bC_S = \frac{1}{2} ( \bC - \bC^*) 
\]
have the same domains and form the self-adjoint (Hermitian) and skew-adjoint (\makered{skew}-Hermitian) part of $\bC$, respectively. If $\bC$ is bounded, then the domains of $\bC_H$ and $\bC_S$ are trivially identical, but
this assumption will also allow us to generalize our  setup to the case of unbounded $\bC$.
\makered{In some situations we prefer to use $\bR$ and $\bJ$ instead of $\bC_H$ and $-\bC_S$ to improve readability of complicated expressions. This is the common notation used for dissipative Hamiltonian systems, see e.g. \cite{JacZ12,SchJ14}.}

For linear evolution equations~\eqref{lineveq} with dissipative operator~$\bA\in\cB(\cH)$, we analyze the short-time behavior 
using the concept of \emph{hypocoercivity}. 
This concept was introduced in \cite{Villani} for the study of evolution equations of this form 
(mostly partial differential equations as well as integro-differential equations like the Boltzmann equation), for which the (possibly unbounded) dissipative operator~$\bA=-\mC$ generates a uniformly exponentially stable $C_0$-semigroup $(e^{-\mC t })_{t\geq 0}$; see e.g.~\cite[Section V.1, Eq.~(1.9)]{EngelNagel2000}.


The original definition of hypocoercivity from \cite{Villani} is as follows:

\begin{definition}\label{def:hypoco}
Let $\bL$ be an (unbounded) operator on a separable Hilbert space $\cH$ generating a strongly continuous semigroup $(e^{-t\bL})_{t \ge 0}$, and let $\widetilde \cH$ be a Hilbert space continuously and densely embedded in $(\ker \bL)^\perp$, endowed with a Hilbertian norm $\norm{\cdot}_{\widetilde \cH}$. The operator $\bL$ is said to be \emph{hypocoercive} on $\widetilde \cH$ if there exists a finite constant $C$ and some $\lambda>0$ such that
\[ \forall h_0 \in \widetilde \cH,\quad \forall t \ge 0:\quad \norm{e^{-t\bL} h_0}_{\widetilde \cH} \le Ce^{-\lambda t}\norm{h_0}_{\widetilde \cH}. \]
\end{definition}

In this article we consider mainly bounded operators $\bL$ with $\ker \bL=\{0\}$. Hence, we typically choose $\widetilde \cH=\cH$, except for the example in Section \ref{sec:lorentz}.

In Definition \ref{def:hypoco} hypocoercivity is characterized in terms of the exponential decay of the semigroup. So, in most of the literature on hypocoercivity, the main goal is the construction of a Lyapunov functional to prove exponential decay of the solution with a realistic rate $\lambda$. That functional (often a perturbation of $\norm{\cdot}_{\widetilde \cH}$ or a natural \emph{energy}) is typically adapted to each problem and application. By contrast, here we assume that the functional setting (i.e., $\widetilde \cH$, $\norm{\cdot}_{\widetilde \cH}$) is given, and we are interested in the short-time behavior of $\norm{x(t)}_{\widetilde \cH}$ and its relation to the hypocoercivity structure of the generator. 

To this end we study equivalent hypocoercivity definitions in terms of the generator in the next section. 
%
Extending from the finite-dimensional setting (as given in \cite{AAC}) to infinite-dimensional Hilbert spaces, it will become necessary to sharpen the conditions known from the matrix case. For this we adapt concepts that were introduced in the context of control theory for (approximate resp. exact) controllability and stabilizability, see \cite{CurtainZwart}, 
to the context of hypocoercivity. See Appendix~\ref{ssec:control} for a detailed discussion of the interrelation of the different concepts.

The paper is organized as follows:    In Section \ref{sec:HC}, we will discuss equivalent conditions for hypocoercivity in the finite-dimensional case and study how they can be extended to arbitrary Hilbert spaces. In Section \ref{sec_hcindex}, these conditions are used to generalize the notion of hypocoercivity index to this setting. In Section \ref{sec_decay}, we will show how the hypocoercivity index leads to quantitative estimates on the short-time and long-time behavior of a hypocoercive system. In Section \ref{sec_staircase}, we establish a staircase form, a useful tool for displaying the structural properties of our systems. In Section \ref{sec:lorentz}, we apply our results to the Lorentz kinetic equation, \makered{where the technical proof of Lemma~\ref{lem1}(c) is deferred to Appendix~\ref{sec:Proof_of_part_c}}. In Section \ref{sec_furtherexamples}, we study some toy examples to further illustrate some finer points of our notions. Finally, in Appendices \ref{sec:appendix} and \ref{ssec:control} we note some further equivalent conditions and connections to linear systems and control theory, respectively.

In particular, our main results (in \S3--\S4) are:
\begin{enumerate}
    \item For bounded accretive generators~$\bC$ we characterize hypocoercivity in terms of a coercivity-type estimate for (the skew-/self-adjoint parts of)~$\bC$: Lemma \ref{lem:Op-Equivalence}, Theorem \ref{cor:exp-stable}.
    \item This coercivity-type condition gives the natural definition of the hypocoercivity index of $\bC$, and it is always finite: Definition \ref{def:op:HC-index}, Proposition \ref{prop:fin-index}.
    \item The hypocoercivity index of~$\bC$ characterizes the (polynomial) short-time decay of the propagator norm~$\|e^{-\bC t}\|$: Theorem \ref{thm:short-t-decay:sandwich}.
\end{enumerate}


\section{Conditions for Hypocoercivity}\label{sec:HC}
In this section we review a number of equivalent conditions that have been derived for hypocoercivity in the \emph{finite-dimensional} case \cite[Proposition 1]{AAC} and extend their equivalence to the infinite dimensional Hilbert space setting.

In the following we use \emph{invariant subspaces} of operators, which are typically assumed to be closed (cf.~\cite[Section 12.27]{rudin} and \cite[Section 2.4]{CurtainZwart}). Despite this convention, to avoid confusion we will use the terminology \emph{closed invariant subspace}.

\begin{proposition}\label{propi}
Let $\bJ \in \cB(\cH)$ be skew-adjoint and $\bR \in \cB(\cH)$ be self-adjoint with $\bR \ge 0$ and let $\dim \ker \bR < \infty$. Then the following conditions are equivalent:

\begin{enumerate}[label=(B\arabic*)]
\item\label{p1} There exists $m \in \N_0$ 
such that
\begin{subequations} 
\begin{equation} \label{p1a}
 \overline{\Span}( \bigcup_{j=0}^m \im ( \bJ^j \sqrt{\bR})) = \mathcal{H}.
\end{equation}
\item\label{p1c}
There exists $m \in \N_0$ such that
\begin{equation} \label{p1b}
 \bigcap_{j=0}^m \ker ( \sqrt{\bR} (\bJ^*)^j) = \{ 0 \}. 
\end{equation}
\item\label{p2} The operators $\bJ$ and $\bR$ satisfy
\[ \exists m \in \N_0\ \forall x \in \cH\setminus\{0\} \ \exists j=j(x) \in\{0,..., m\}:  \, 
\langle \bJ^j \bR (\bJ^*)^j x, x \rangle > 0. \]
\end{subequations}
This means that $\sum_{j=0}^m \bJ^j \bR (\bJ^*)^j > 0$.
\item\label{p3} If a closed subspace $V$ of $\ker \bR$ is invariant under $\bJ$ then $V = \{0\}$.
\item\label{p4} No eigenvector of $\bJ$ lies in the kernel of $\bR$.
\end{enumerate}
If one of \ref{p1}--\ref{p2} holds for a particular $m$, all of them hold for the same $m$. Moreover, these conditions then also hold for all $m' > m$, and the minimum of all $m$ such that \ref{p1}--\ref{p2} hold is less than or equal to $\dim \ker \bR$.
\end{proposition}

\begin{proof}
We assume $\dim\ker\bR \ge 1$, as otherwise the equivalence of \ref{p1}--\ref{p4} is trivial. 
For the equivalence of \ref{p1} and \ref{p1c}, note that a subspace of $\cH$ is dense if and only if its orthogonal complement is the trivial subspace $\{0\}$. Hence, \ref{p1} is equivalent to
\begin{align*}
\{0\} &= \left( \Span \left( \bigcup_{j=0}^m \im ( \bJ^j \sqrt{\bR} ) \right) \right)^\perp = \left( \bigcup_{j=0}^m \im ( \bJ^j \sqrt{\bR} ) \right)^\perp \\
&= \bigcap_{j=0}^m \left( \im ( \bJ^j \sqrt{\bR} ) \right)^\perp = \bigcap_{j=0}^m \ker ( \sqrt{\bR} (\bJ^*)^j).
\end{align*}
For \ref{p1c} $\Leftrightarrow$ \ref{p2} we note that for $x \in \cH$,
\[ \sqrt{\bR} (\bJ^*)^j x = 0 \Longleftrightarrow \norm{\sqrt{\bR} (\bJ^*)^j x}^2 = \langle \sqrt{\bR} (\bJ^*)^j x, \sqrt{\bR} (\bJ^*)^j x \rangle = \langle \bJ^j \bR (\bJ^*)^j x, x \rangle = 0. \]
Hence, if $\bigcap_{j=0}^m \ker ( \sqrt{\bR} (\bJ^*)^j) = \{0\}$ holds for some $m \in \N_0$  
and we suppose that $\forall j \le m$ 
$\langle \bJ^j \bR (\bJ^*)^j x, x \rangle = 0$, then $x=0$ by assumption. Conversely, if \ref{p2} holds with some $m \in \N_0$ and $\sqrt{\bR} (\bJ^*)^j x = 0$ for all $j \le m$ then $x=0$ follows.
This proof also shows that the smallest value for $m$ is equal in \ref{p1}, \ref{p1c}, and \ref{p2}. 

For \ref{p2} $\Rightarrow$ \ref{p3}, if $0 \ne x \in \ker \bR$ then for some $j \in \N_0$, $\sqrt{\bR} (\bJ^*)^j x \ne 0$, so $\bJ^j x = (-1)^j (\bJ^*)^j x \not\in \ker \sqrt{\bR} = \ker \bR$. 
Now let this $x\in V\subset \ker \bR$. But then $V$ cannot be $\bJ$-invariant.

For \ref{p3} $\Rightarrow$ \ref{p2}, we show the claim for $m = \dim \ker \bR$, which also shows that the smallest value for $m$ is less than or equal to $\dim \ker \bR$. Suppose $x \in \cH$ and $\langle \bJ^j \bR (\bJ^*)^j x, x \rangle = 0$, i.e., $\sqrt{\bR} (\bJ^*)^j x = 0$ and hence $\bJ^j x \in \ker \bR$ $\forall j=0,\dotsc,m$. As $\dim \ker \bR = m$ there exists $q \in \{1, \dotsc, m\}$ such that $\bJ^q x \in \Span\{x, \bJ x, \dotsc, \bJ^{q-1} x \}$ which implies that $\Span \{ x, \dotsc, \bJ^{q-1} x\}$ is a closed $\bJ$-invariant subspace of $\ker \bR$ which equals $\{0\}$ by assumption and hence $x=0$.

Finally, we establish \ref{p3} $\Leftrightarrow$ \ref{p4}. First, assuming that some eigenvector $x \ne 0$ of $\bJ$ satisfies $\bR x = 0$, this eigenvector spans a nonzero subspace of $\ker \bR$ which is invariant under $\bJ$. This gives $\neg$\ref{p4} $\Rightarrow$ $\neg$\ref{p3}, i.e., \ref{p3} $\Rightarrow$ \ref{p4}. Conversely, suppose that there exists a nonzero $\bJ$-invariant subspace of $\ker \bR$. As $\ker \bR$ is assumed to be finite-dimensional, it contains an eigenvector of $\bJ$ because the restriction of $\bJ$ to this subspace is still skew-adjoint.
This shows $\neg$\ref{p3} $\Rightarrow$ $\neg$\ref{p4}, i.e., \ref{p4} $\Rightarrow$ \ref{p3}.
\end{proof}

In many practically relevant examples of hypocoercive evolution equations, the self-adjoint part of their generator has an infinite dimensional kernel. In kinetic theory this kernel is typically spanned by the \emph{local equilibria}, see \cite{Villani, AAC}. However, when appropriately decomposing the original problem into subproblems, the analogous kernel is only finite dimensional and spanned by the \emph{global equilibria} (of the subproblems); for details see \cite{AAC, AAM22Oseen} and the example of Section \ref{sec:lorentz}, below.

To address this situation, we now discuss  the conditions in Proposition~\ref{propi} for the case $\dim \ker \bR = \infty$.

\begin{proposition}\label{propii}
Let $\bJ \in \cB(\cH)$ be skew-adjoint and $\bR \in \cB(\cH)$ be self-adjoint with $\bR \ge 0$ and possibly $\dim \ker \bR = \infty$. Consider the following modifications of conditions  \ref{p1}--\ref{p4} in Proposition~\ref{propi}:

\begin{enumerate}[label=(B\arabic*')]
\item\label{p1'} There exists $m \in \N_0\cup\{\infty\}$ 
such that
\begin{subequations} 
\begin{equation} \label{p1a'}
 \overline{\Span}( \bigcup_{j=0}^m \im ( \bJ^j \sqrt{\bR})) = \mathcal{H}.
\end{equation}
\item \label{p1c'}
There exists $m \in \N_0 \cup \{\infty\}$ such that
\begin{equation} \label{p1b'}
 \bigcap_{j=0}^m \ker ( \sqrt{\bR} (\bJ^*)^j) = \{ 0 \}. 
\end{equation}
\item\label{p2'} There exists $m \in \N_0\cup\{\infty\}$ such that $\forall x \in \cH \setminus\{0\}$ $\exists j=j(x) \le m$ if $m$ is finite or $j=j(x) \in \N$ if $m = \infty$ such that
\[ 
\langle \bJ^j \bR (\bJ^*)^j x, x \rangle > 0. \]
\end{subequations}
\item\label{p3'} If a closed subspace $V$ of $\ker \bR$ is invariant under $\bJ$ then $V = \{0\}$.
\item\label{p4'} No eigenvector of $\bJ$ lies in the kernel of $\bR$.
\end{enumerate}
Then the implications \ref{p1'} $\Leftrightarrow$ \ref{p1c'} $\Leftrightarrow$ \ref{p2'} $\Leftrightarrow$ \ref{p3'} $\Rightarrow$ \ref{p4'} hold. 

The implication \ref{p4'} $\Rightarrow$ \ref{p3'} holds if and only if $\dim \ker \bR < \infty$.
\end{proposition}
\begin{proof} The proof is essentially the same as  that of Proposition~\ref{propi}, replacing $j \le m$ by $j \in \N_0$ where necessary. For the implication \ref{p3'} $\Rightarrow$ \ref{p2'} the proof needs the following modification: For the considered $x$, let
\[
  V\coleq \Span \left( \bigcup_{j=0}^\infty \bJ^jx\right) \subset \ker \bR.
\]
Then 
\[
  \bJ(V)=\Span \left( \bigcup_{j=1}^\infty \bJ^jx\right) \subset V,
\]
and hence $\overline V\subset \ker\bR$ is a $\bJ$-invariant subspace of $\ker\bR$ which equals $\{0\}$ by assumption.

 To show that $\dim \ker \bR < \infty$ is necessary for \ref{p4'} $\Rightarrow$ \ref{p3'}, consider $\cH \coleq L^2([0,1])$ over $\C$, $\bR=0$ and let $\bJ$ be the multiplication operator $(\bJ f)(x) = i x f(x)$ for $f \in \cH$ and $x \in [0,1]$. Then $\cH$ is a $\bJ$-invariant subspace of $\ker \bR = \cH$, so \ref{p3'} fails to hold. But as $\bJ$ has no eigenvectors at all, \ref{p4'} holds.
\end{proof}

\begin{remark}
Conditions similar to the ones given in Proposition~\ref{propii}~\ref{p1'}--\ref{p3'} characterize \emph{asymptotically controllable/observable} linear systems, see~\cite[\S6]{CurtainZwart} and Appendix~\ref{ssec:control}.
\end{remark}

In the finite-dimensional setting, the conditions listed in Proposition~\ref{propi} are equivalent to hypocoercivity and thus exponential decay of the dynamical system \cite[Lemma 2.4]{AAC}. But, in the infinite-dimensional case, these conditions are too weak, as the following (strictly accretive) example shows:  Consider $\cH = \ell^2(\N)$ with canonical basis $(b_n)_{n \in \N}$ and operators $\bJ=0$, $\bR b_n = b_n/n$ (i.e., $\bR = \diag(1/n : n \in \N)$). Then $e^{-\bC t} b_n = e^{-t/n} b_n$ has arbitrarily small decay rate $1/n = \lambda >0$, hence $\bC$ is not hypocoercive. 
This illustrates that the absence of purely imaginary eigenvalues alone is not enough to characterize hypocoercivity of an accretive operator.
In an infinite-dimensional setting the spectrum of an operator may contain also an essential spectrum. This may lead to additional difficulties in establishing hypocoercivity, see e.g. \cite{EdmE18} for a detailed discussion of essential and point spectra.

A natural modification to prevent this problem is to demand some uniformity in Proposition \ref{propi} \ref{p2}, e.g., to replace ``$>0$'' by ``$\ge \kappa \bI$ for some $\kappa>0$.'' In this spirit, the following lemma is a stricter variant of Proposition \ref{propi} \ref{p1} and \ref{p2}. 
In fact, as we will see in Theorem \ref{cor:exp-stable} below, these new conditions are equivalent to hypocoercivity for bounded accretive operators.

\begin{lemma}[{\cite[Lemma 2]{AAM22Oseen}}] \label{lem:Op-Equivalence}
Let $\bJ \in \cB(\cH)$ be skew-adjoint and $\bR \in \cB(\cH)$ be self-adjoint with $\bR \ge 0$.
Then the following 
conditions are equivalent:
\begin{enumerate}[label=(B\arabic*'')]
\setcounter{enumi}{0}
\item \label{B:G*_surjective}
There exists $m\in\N_0$ such that
\[
\Span \Bigg(\bigcup_{j=0}^m \range\big(\mJ^j \sqrt{\bR}\big)\Bigg) =\cH\ .
\]
\item \label{B:KRC'op}
There exists $m\in\N_0$ such that
\[
 \bigcap_{j=0}^m \ker\big(\sqrt{\bR}(\mJ^*)^j\big) 
= \{0\} \quad
\text{and}\quad  
 \Span \Bigg(\bigcup_{j=0}^m \range\big(\mJ^j \sqrt{\bR}\big)\Bigg)\ \text{ is closed.}
\] 

\item \label{B:Tmop}
There exists $m\in\N_0$ such that $\sum_{j=0}^m \mJ^j \mR (\mJ^*)^j \ge\kappa\mI$ holds for some $\kappa>0$. 
\end{enumerate}
 
Moreover, the smallest possible~$m\in\N_0$ coincides in all cases (if it exists).
\end{lemma}

\begin{remark}\label{remark25}
\begin{enumerate}[label=(\roman*)]
\item\label{remark25.1}  
Contrarily to the case of Proposition \ref{propi}, the dimension of $\ker \bR$ does \emph{not} give an upper bound on the minimal value of $m$ such that \ref{B:Tmop} holds, as the following example shows: Consider the block-diagonal operator
    \[ \bC = \diag \left\{ \begin{bmatrix} \frac{1}{n} & 1 \\ -1 & 1 \end{bmatrix} : n \in \N \right\} = \bR - \bJ \] 
    with
    \[ \bR = \diag \left\{ \begin{bmatrix} \frac{1}{n} & 0 \\ 0 & 1 \end{bmatrix} : n \in \N \right \},\quad \bJ = \diag \left\{ \begin{bmatrix} 0 & -1 \\ 1 & 0 \end{bmatrix} : n \in \N \right\}. \]
    Then, although $\ker \bR = \{0\}$, \ref{B:Tmop} fails to hold for $m = 0$ but it holds for $m=1$ with $\kappa=1$, as is seen in
    \[ \bR + \bJ \bR \bJ^* = \diag \left\{ \begin{bmatrix} 1+\frac{1}{n} & 0 \\ 0 & 1 + \frac{1}{n} \end{bmatrix} : n \in \N \right\} \ge \bI. \]
\item We could also formulate conditions \ref{B:G*_surjective}--\ref{B:Tmop} for $m = \infty$ (scaling the operators such that the series in \ref{B:Tmop} converges, if necessary), but we will see in Proposition \ref{prop:fin-index} below that this case cannot occur in the setting of bounded operators.
\item \label{remark25.3} We note that if $\bR$ is compact and $\cH$ is infinite-dimensional then~\ref{B:Tmop} cannot hold for any $m$: 
On separable infinite-dimensional Hilbert spaces, the subset of compact operators $\mathcal{K}(\cH)\subset\cB(\cH)$ is an ideal in~$\cB(\cH)$.
Consider $\bC=\bR-\bJ\in\cB(\cH)$ under the assumptions of Lemma~\ref{lem:Op-Equivalence} and $\bR\in\cK(\cH)$.
Hence, for $m\in\N$, the operators $\sum_{j=0}^m \mJ^j \mR (\mJ^*)^j$ are compact. Therefore, condition~\ref{B:Tmop} holds if and only if $\cH$ is finite-dimensional.

Consider $\bC\in\cB(\cH)$ with $\cH=\ell^2(\N)$ given as
\[
 \mC =\mR -\mJ 
\qquad\text{with }
 \mR =\diag(1/n : n \in \N) , \quad
 \mJ =\tridiag(1,0,-1) .
\]
The operator $\mC$ is strictly accretive, since~$\mCH=\mR>0$. Moreover, $\mR$ is compact, hence, the operator~$\bC$ cannot be hypocoercive.
\end{enumerate}
\end{remark}

The next lemma gives additional variants of these conditions that will be useful in the subsequent analysis.
For its proof we note that, 
if $(\cH_i)_{i=1, \dotsc, n}$ and $(\cK_j)_{j=1, \dotsc, m}$ are two finite families of Hilbert spaces with canonical projections $\pi_{j_0} \colon \bigoplus_{j=1}^m \cK_j \to \cK_{j_0}$ and injections $\iota_{i_0} \colon \cH_{i_0} \to \bigoplus_{i=1}^n \cH_i$, respectively, then the kernel and the image of an operator 
\[ \bT \colon \bigoplus_{i=1}^n \cH_i \to \bigoplus_{j=1}^m \cK_j \]
can be written as
\[ \ker \bT = \bigcap_{j=1}^m \ker (\pi_j \circ \bT), \qquad \im \bT = \Span \left( \bigcup_{i=1}^n \im (\bT \circ \iota_i) \right). \]

\begin{lemma}\label{lemmaiii}
Let $\bJ \in \cB(\cH)$ be skew-adjoint and $\bR \in \cB(\cH)$ be self-adjoint with $\bR \ge 0$ and let $m \in \N_0$. Set $\bC = \bR -\bJ \in \cB(\cH)$. Define $\bC_0 := \sqrt{\bR}$ and the iterated commutators $\bC_{j+1} := [\bJ, \bC_j]=\bJ \bC_j-\bC_j\bJ$, $j \in \N_0$. 
Then
\begin{align}
\label{lemmaiii_1}\exists \kappa_1 >0: \sum_{j=0}^m \bJ^j \bR (\bJ^*)^j \ge \kappa_1 \bI & \Longleftrightarrow \Span \Big(\bigcup_{j=0}^m \im ( \bJ^j \sqrt{\bR})\Big) = \cH \\
\Longleftrightarrow \bigcap_{j=0}^m \ker & \big( \sqrt{\bR} (\bJ^*)^j\big) = \{0 \} \textrm{ and } \Span \Big(\bigcup_{j=0}^m \im ( \bJ^j \sqrt{\bR})\Big) \textrm{ is closed}, \nonumber \\
\exists \kappa_2 >0: \sum_{j=0}^m \bC^j \bR (\bC^*)^j \ge \kappa_2 \bI & \Longleftrightarrow \Span \Big(\bigcup_{j=0}^m \im ( \bC^j \sqrt{\bR} )\Big)= \cH \label{lemmaiii_2} \\
\Longleftrightarrow \bigcap_{j=0}^m \ker & \big(\sqrt{\bR} (\bC^*)^j \big)= \{0 \} \textrm{ and } \Span \Big(\bigcup_{j=0}^m \im ( \bC^j \sqrt{\bR} )\Big) \textrm{ is closed}, \nonumber\\
\exists \kappa_3 >0: \sum_{j=0}^m (\bC^*)^j \bR \bC^j \ge \kappa_3 \bI & \Longleftrightarrow \Span \Big(\bigcup_{j=0}^m \im ( (\bC^*)^j \sqrt{\bR} )\Big)= \cH\label{lemmaiii_3} \\
\Longleftrightarrow \bigcap_{j=0}^m \ker & \big(\sqrt{\bR} \bC^j\big) = \{0 \} \textrm{ and } \Span \Big(\bigcup_{j=0}^m \im ( (\bC^*)^j \sqrt{\bR} )\Big) \textrm{ is closed}, \nonumber\\
\exists \kappa_4 >0: \sum_{j=0}^m \bC_j^* \bC_j \ge \kappa_4 \bI & \Longleftrightarrow \Span \Big(\bigcup_{j=0}^m \im \bC_j\Big) = \cH \label{lemmaiii_4} \\
\Longleftrightarrow \bigcap_{j=0}^m \ker & (\bC_j) = \{0 \} \textrm{ and } \Span \Big(\bigcup_{j=0}^m \im \bC_j\Big) \textrm{ is closed}.\nonumber
\end{align}
In addition, all listed conditions are equivalent. 
\end{lemma}
\begin{proof}
Considering \eqref{lemmaiii_1} for instance, if we set
\[ \bT = (\sqrt{\bR}, \sqrt{\bR}\bJ^*, \dotsc, \sqrt{\bR} (\bJ^*)^m) \colon \cH \to \cH \oplus \dotsc \oplus \cH \]
then the assumption $\sum_{j=0}^m \bJ^j \bR (\bJ^*)^j \ge \kappa_1 \bI$ means precisely that $\norm{\bT x}^2 \ge \kappa_1 \norm{x}^2$. An application of \cite[Theorem 2.21]{Brezis} together with \cite[Theorem 2.19]{Brezis} then immediately gives the claim. The other equivalences are shown in the same way.

To prove the equivalence of \eqref{lemmaiii_1}--\eqref{lemmaiii_4}, we note that the conditions on the intersections of the kernels are equivalent because of Lemma \ref{lemmaii}. For the equivalence of the closedness of the given linear spans in rows one and two, we first show that $\im (\bJ^k \sqrt{\bR}) \subseteq \Span \big(\bigcup_{j=0}^m \im (\bC^ j \sqrt{\bR})\big)$ for $k \in \{0, \dotsc, m\}$. To see this, we note that for any $x \in \cH$
\begin{align*}
\bJ^k \sqrt\bR x &= (\bR-\bC)^k \sqrt\bR x \
\end{align*}
consists of a sum of $2^k$ terms of the form
\[ \bT_1 \dotsc \bT_k \sqrt\bR x \]
with $\bT_i \in \{ \bR, -\bC \}$ for all $i \in \{ 0, \dotsc, k \}$. Each of these terms can be written in the form $\bC^j \sqrt \bR y$ for some $y \in \cH$ and $j \in \{0, \dotsc, m\}$. The converse inclusion, $\im(\bC^k \sqrt \bR) \subseteq \Span \big(\bigcup_{j=0}^m \im(\bJ^j \sqrt\bR)\big)$, is seen by expanding $\bC^k \sqrt\bR x = (\bR - \bJ)^k \sqrt\bR x$ and using the same argument again. As $\bC^* = \bR + \bJ$ we similarly obtain equivalence of \eqref{lemmaiii_1} and \eqref{lemmaiii_3}. To see equivalence of \eqref{lemmaiii_1} and \eqref{lemmaiii_4} we note that $\bC_k$ consists of a sum of terms of the form
\[ \pm \bT_0 \dotsc \bT_k \]
where exactly one of the $\bT_i$ equals $\sqrt\bR$ and the others equal $\bJ$, hence we have $\im(\bC_k) \subseteq \Span \bigcup_{j=0}^k \im (\bJ^j \sqrt \bR)$.

Finally, we establish $\im(\bJ^k \sqrt\bR) \subseteq \Span \big(\bigcup_{j=0}^k \im \bC_j\big)$ inductively. This is clear for $k=0$ because $\bC_0 = \sqrt\bR$. Assuming that the claim holds for $k$, we see that for any $x \in \cH$ there are $x_j \in \cH$ ($j = 0, \dotsc, k$) such that
\begin{align*}
\bJ^{k+1} \sqrt\bR x &= \bJ \bJ^k \sqrt\bR x = \bJ \sum_{j=0}^k \bC_j x_j\\
&= \sum_{j=0}^k ( \bJ \bC_j - \bC_j \bJ + \bC_j \bJ) x_j \\
&= \sum_{j=0}^k \bC_{j+1} x_j + \sum_{j=0}^k \bC_j \bJ x_j \subseteq \Span \Big(\bigcup_{j=0}^{k+1} \im \bC_j\Big). \qedhere
\end{align*}
\end{proof}

The middle column of these equivalences (i.e., the statement that the images of $\bJ^j \sqrt{\bR}$ etc.\ span $\cH$) can be considered as variants of the classical \emph{Kalman rank condition} in control theory. In fact,
conditions similar to the ones given in Lemma~\ref{lem:Op-Equivalence}~\ref{B:G*_surjective}--\ref{B:Tmop} characterize \emph{exactly controllable/observable} linear systems, as is detailed in \cite[\S6]{CurtainZwart} and Appendix~\ref{ssec:control}.


\section{Hypocoercivity Index}\label{sec_hcindex}

In this section we define the hypocoercivity index for bounded accretive operators in the form appropriate for the (typically infinite-dimensional) Hilbert space setting.

\begin{definition}[{\cite[Definition 5]{AAM22Oseen}}]
\label{def:op:HC-index}
Let~$\cH$ be a separable Hilbert space.
For accretive operators $\mC\in\cB(\cH)$, the~\emph{hypocoercivity index (HC-index)~$m_{HC}$ of~$\mC$} is defined as the smallest integer~$m\in\N_0$ (if it exists) such that 
\begin{equation}\label{Op:hypocoercive:kappa}
  \sum_{j=0}^m (\mC^*)^j \mCH \mC^j \ge\kappa\mI 
\end{equation}
for some $\kappa>0$. 
\end{definition}

\begin{remark}
\begin{enumerate}[label=(\roman*)]
    \item
Using the equivalence of conditions from Lemma \ref{lemmaiii},  Definition~\ref{def:op:HC-index} could also be based on the coercivity of $\sum_{j=0}^m \mCS^j \mCH (\mCS^*)^j$, as used in Lemma \ref{lem:Op-Equivalence} \ref{B:Tmop}.


\item Following Proposition \ref{propii}, conditions \ref{B:G*_surjective} and \eqref{Op:hypocoercive:kappa}, the latter if $\|\bC\|<1$, would also make sense for $m=\infty$. Hence, a natural question is whether an infinite hypocoercivity index would make sense in Definition \ref{def:op:HC-index}. Indeed, this may be the case for \emph{unbounded} operators: 
Example \ref{example:AAC} below presents an unbounded hypocoercive operator for which no (finite) hypocoercivity index exists. 
\end{enumerate}
\end{remark}

For \emph{bounded} operators, the hypocoercivity index can only be finite, as the following result shows.
\begin{proposition}\label{prop:fin-index}
Let $\bJ \in \cB(\cH)$ be skew-adjoint and $\bR \in \cB(\cH)$ be self-adjoint with $\bR \ge 0$.
If
\begin{equation} \label{span:infinity}
 \Span \left( \bigcup_{j=0}^\infty \im \left( \bJ^j \sqrt{\bR}\right)\right) = \cH ,
\end{equation}
then there exists $m \in \N_0$ such that
\begin{equation} \label{span:m}
\Span \left( \bigcup_{j=0}^m \im \left( \bJ^j \sqrt{\bR}\right)\right) = \cH. 
\end{equation}
\end{proposition}

\begin{proof}
We have
\[ \cH = \bigcup_{j=0}^\infty \left( \Span \bigcup_{i=0}^j \im \left( \bJ^i \sqrt{\bR} \right) \right) = \bigcup_{j=0}^\infty \im ( \bA_j), \]
where $\bA_j \colon \cH^{j+1} \to \cH$ is defined by 
\[ \bA_j \coleq \mathfrak{S} \circ \left( \sqrt{\bR} \oplus \bJ \sqrt{\bR} \oplus \dotsc \oplus \bJ^j \sqrt{\bR} \right). \]
Here, $\sqrt{\bR} \oplus \bJ \sqrt{\bR} \oplus \dotsc \oplus \bJ^j \sqrt{\bR} \colon \cH^{j+1} \to \cH^{j+1}$ is the usual direct sum of the operators $\bJ^i \sqrt{\bR}$, $i=0, \dotsc ,j$, and $\mathfrak{S} \colon \cH^{j+1} \to \cH$, $(x_0, \dotsc, x_j) \mapsto x_0 + \dotsc + x_j$ the sum of all components.

The following argument is inspired by \cite[Theorem 1.2]{Fuhrmann}: 
Suppose each $\im(\bA_j)$ is meager, i.e., a countable union of nowhere dense sets, then the same is true for $\cH$, but as a complete metric space $\cH$ cannot be meager by Baire's theorem \cite[\nopp 2.2]{rudin}. Hence there is some $m \in \N_0$ such that $\im(\bA_m)$ is nonmeager and by the open mapping theorem \cite[\nopp 2.11]{rudin} $\bA_m$ is surjective, which gives the claim.

\end{proof}


\section{Short-time decay of the propagator norm}\label{sec_decay}

Coercive operators $\bC$ give rise to exponential decay of all solutions with the multiplicative constant $C=1$ in Definition \ref{def:hypoco}. But for hypocoercive operators, in general, one has to choose $C>1$, since an initial condition $x_0\in\ker \bR$ implies for the corresponding solution of \eqref{lineveq}: $\frac{d}{dt}\|x(t)\|\Big|_{t=0}=0$. This behavior is well-known analytically and numerically for degenerate kinetic PDEs \cite{AE14, FMP06} as well as for ODEs~\cite{Da07,So06,TrEm05}. In the latter case the precise behavior of the propagator norm close to $t=0$ (and hence also of the solution norms) was recently characterized in \cite{AAC2} in terms of the hypocoercivity index of $\bC$. 

In this section we extend that analysis to evolution equations of the form \eqref{lineveq} in infinite-dimensional separable Hilbert spaces. Typically, such an analysis of long- and short-time decay is carried out via the study of the semigroup associated with the solution of~\eqref{lineveq}.
Recall that a $C_0$-semigroup $(\bP(t))_{t\ge0}$ is called \emph{uniformly stable} if $\lim_{t\to\infty} \|\bP(t)\|=0$, and \emph{uniformly exponentially stable} if there exists $\varepsilon>0$ such that $\lim_{t\to\infty} e^{\varepsilon t}\|\bP(t)\|=0$.

For strongly continuous semigroups, uniform exponential stability can be characterized in various ways, e.g., as in Proposition~\ref{prop:UniExpStableSG} below or via an integrability condition on the propagator norm, see~\cite[Theorem V.1.8]{EngelNagel2000} and \cite[Lemma 4.1.2]{CurtainZwart}.
In the Hilbert space setting, uniform exponential stability of a strongly continuous semigroup can be inferred from its generator: either as a condition on its resolvent, see~\cite[Theorem V.1.11]{EngelNagel2000}; 
or by the existence of a strict Lyapunov functional, see, e.g.,~\cite[\nopp 4.1.3]{CurtainZwart}.
%
In our analysis based on the concept of hypocoercivity, to decide if an accretive operator~$\mC\in\cB(\cH)$ is hypocoercive, we will study its propagator norm~$\|e^{-\mC t}\|$ for \emph{small} $t\geq 0$. 

\subsection{Main results}\label{sec:4-main}

In general, we have limited information about differentiability or analyticity of the propagator norm $\norm{e^{-\bC t}}$ and hence about its initial behavior:
For bounded operators~$\bC\in\cB(\cH)$, the propagator norm~$\|e^{-\bC t}\|$ is right-differentiable at $t=0$, since norms are differentiable in every point and every direction~\cite[Chapter 2, Example 7.7]{De85}.
The one-sided differentiability is used to define the~\emph{logarithmic norm} which determines the exponential rate of the short-time behavior of~$\|e^{-\bC t}\|$, see~\cite{So06,TrEm05}.
In this section, we use methods introduced in \cite{AAC2} to show that the hypocoercivity index of an accretive operator~$\mC\in\cB(\cH)$ characterizes the initial decay at $t=0$ of the propagator norm.

We denote the solution semigroup pertaining to~\eqref{lineveq} by $\mP(t)\coleq e^{-\mC t}\in \cB(\cH)$ ($t \ge 0$) and its operator norm by $\|\mP(t)\| \coleq \sup\{\|\mP(t)x\| : \|x\| =1\}$. 
Our main result relating the qualitative short-time decay behavior to the hypocoercivity index is as follows.

\begin{theorem} \label{thm:short-t-decay:sandwich}
Let the operator $\bC\in\cB(\cH)$ be accretive.
\begin{itemize}
    \item[(a)] If $\bC$ has hypocoercivity index $\mHC\in\N_0$, then there exist constants $c_1\geq c \ge c_2>0$, $C_1, C_2 \ge 0$ and $t_0>0$ such that with $a:=2\mHC+1$ we have the estimates
\begin{equation} \label{short-t-decay:sandwich}
 1-c_1 t^a -C_1 t^{a+1}
\leq \| \mP(t)\|
\leq 1-c_2 t^a + C_2 t^{a+1}
\quad\text{for } t \in [0, t_0]
\end{equation}
and
\begin{equation}\label{d25}
\norm{\mP(t)} = 1 - ct^a + o(t^a)\qquad \mbox{for }t\to0,
\end{equation}
where
\small
\begin{equation} \label{def-c}
  c \coleq
 \frac{\lim\limits_{\delta\to 0}\ \Big[ \inf \big\{\norm{\sqrt{\bC_H}\bC^{m_{HC}} x}^2 : \\
 x\in\sphere,\ \norm{\sqrt{\bC_H}\bC^j x}^2
\le \delta\ \forall j=0,\dotsc,m_{HC}-1 \big\}\, \Big]}{(2m_{HC}+1)!\, \binom{2m_{HC}}{m_{HC}}}\,, 
\end{equation}
with the sphere $\sphere \coleq \{ x\in\cH : \norm{x} =1\}$.
\normalsize  
\item[(b)] Conversely, if the propagator norm satisfies \eqref{d25} with some real constants $c>0$ and $a>0$, then $a$ is an odd integer 
and $\bC$ is hypocoercive with index $\mHC=\frac{a-1}{2}\in\N_0$.
\end{itemize}
\end{theorem}

\makered{
\begin{remark}
In \cite{GaMi13} the short time decay behavior of a kinetic Fokker--Planck equation on the torus in $x$ was computed as $1-\frac{t^3}{12}+o(t^3)$. 
Again, in Fourier space and by using a Hermite function basis in velocity, this model can be written as an (infinite dimensional) conservative-dissipative system with hypocoercivity index 1 (see \S2.1 of \cite{GaMi13}). 
In that paper it was also mentioned that the decay exponent in~\eqref{short-t-decay:sandwich} can be seen as some ``order of hypocoercivity'' of the generator.
\end{remark}
}

We will split the proof into several parts: The upper and lower bounds in \eqref{short-t-decay:sandwich} will be proven in \S\ref{sec:4-upper} and \S\ref{sec:4-lower}, respectively, and the reverse direction (b) will be proven in \S\ref{sec:reverse}. First we shall make some remarks and discuss consequences of these decay bounds.

\begin{remark}\label{remark_unterschied}
\begin{enumerate}[label=(\roman*)]
\item As can be seen from the proofs below (in particular, inequality \eqref{est:g:C0:final}), the constants $c_1, c_2$ in \eqref{short-t-decay:sandwich} depend on the time interval $t \in [0,t_0]$ and they can be made arbitrarily close to each other (and hence to $c$) for $t_0 \to 0$. Because we have no information about the speed of convergence of $c_1,c_2$ to $c$ (as $t_0\to0$), the remainder in \eqref{d25} is only of the form $o(t^a)$, instead of $\bigO(t^{a+1})$.
\item In the finite dimensional ODE-case, $\|e^{-\bC t}\|$ is analytic on some (short) time interval $[0,t_1)$, see \cite{AAC2}. However, in the infinite dimensional case 
this is not true in general, as was recently illustrated in a counterexample by G.M.\ Graf \cite{Graf}. 
But in concrete examples where $\|\mP(t)\|$ is analytic on $[0,t_0)$ for some $t_0>0$ or just $a$ times differentiable at $t=0$, then~\eqref{short-t-decay:sandwich} would already imply that the propagator norm satisfies \eqref{d25} for $t\in[0,t_0]$ (and not just in the limit $t\to 0$), and some $c\in\R$ such that $0 <c_2 \leq c \leq c_1$.  
\item \label{remark_unterschied_3} Our proofs below are based on the corresponding results in \cite{AAC2}, but with a significant technical difference: In the finite dimensional case, existence of the hypocoercivity index $m_{HC}$ of an operator $\bC$ means that the operator on the left-hand side of \eqref{Op:hypocoercive:kappa} for $m=m_{HC}-1$ has a non-trivial kernel (because $m_{HC}$ is defined as the \emph{minimal} value of $m$ such that \eqref{Op:hypocoercive:kappa} is satisfied). But in the infinite-dimensional setting, the failure of the coercivity estimate \eqref{Op:hypocoercive:kappa} to hold for this $m=m_{HC}-1$ translates only into the weaker condition:
\begin{equation}
\label{HC_failure}
\forall \e>0\ \exists x_0=x_0(\e) \in \cH\setminus\{0\}\ \forall j = 0,\dotsc,m_{HC}-1: \norm{\sqrt{\bC_H} \bC^j x_0} \le \e \norm{x_0}\,. 
\end{equation}
Thus, Lemma \ref{refined} below will deal with initial conditions $x_0$ such that $\norm{\sqrt{\bC_H} \bC^j x_0}$ is small for $j=0,...,\mHC-1$, while the analogous (and technically simpler)  Lemma A.4 in \cite{AAC2} may assume that all these $\sqrt{\bC_H} \bC^j x_0$ vanish.
\end{enumerate}
\end{remark}

In the terminology of semigroup theory, an operator $\mC\in\cB(\cH)$ is \emph{hypocoercive} if it generates a uniformly exponentially stable semigroup $\big(e^{-\mC t}\big)_{t\geq 0}$.
Recall the following 
characterizations of uniformly exponentially stable semigroups:
\begin{proposition} [{\cite[Proposition V.1.7]{EngelNagel2000}}]
\label{prop:UniExpStableSG}
For a strongly continuous semigroup~$\big(\mP(t)\big)_{t\geq 0}$, the following assertions are equivalent: 
\begin{enumerate}[label=(\alph*)]
 \item 
$\omega_0 \coleq \displaystyle \inf \big\{\omega\in\R\:\big|\: \lim_{t\to\infty} e^{-\omega t}\|\mP(t)\| =0 \big\} <0$, i.e., $\big(\mP(t)\big)_{t\geq 0}$ is uniformly exponentially stable.
 \item
$\lim_{t\to\infty} \|\mP(t)\| =0$, i.e., $\big(\mP(t)\big)_{t\geq 0}$ is uniformly stable.
 \item \label{prop:UniExpStableSG:c}
$\|\mP(t_0)\|<1$ for some $t_0>0$. 
\end{enumerate}
\end{proposition}
\begin{remark}\label{UC_SG}
For a uniformly continuous semigroup~$(\mP(t))_{t\geq 0}$ and its (bounded) generator~$\bA$ one has $\omega_0=s(\bA)$, where $s(\bA):=\sup\{\Re\lambda:\ \lambda\in\sigma(\bA)\}$ is called the \emph{spectral bound of~$\bA$}, see e.g.~\cite[Corollary IV.2.4]{EngelNagel2000}. \end{remark}
Combining Theorem~\ref{thm:short-t-decay:sandwich}, Proposition~\ref{prop:UniExpStableSG}(c), and Lemma~\ref{lem:fin-index} yields the following key result:

\begin{theorem}\label{cor:exp-stable} 
Let the operator~$\mC\in\cB(\cH)$ be accretive. Then $\mC$ is hypocoercive (or, equivalently, it generates a uniformly exponentially stable $C_0$-semigroup~$\big(e^{-\mC t}\big)_{t\geq 0}$) if and only if it has a hypocoercivity index $\mHC\in\N_0$ (see Definition~\ref{def:op:HC-index} / satisfying one of the conditions in Lemma \ref{lem:Op-Equivalence}).
\end{theorem}

\subsection{Upper bound of the propagator norm}\label{sec:4-upper}

For establishing the upper and lower bounds of Theorem \ref{thm:short-t-decay:sandwich} we introduce some notation which will be used in the proofs below. Given any fixed $\bC \in \cB(\cH)$, we set
\begin{align}\label{A:Q}
\mQ(t)
& \coleq e^{-\mC^* t}e^{-\mC t} = \sum_{j=0}^{\infty} \frac{t^j}{j!}\ \mU_j \qquad \textrm{for }t \ge 0, \\
\textrm{with }\quad \mU_j & =(-1)^j \sum_{k=0}^j \binom{j}{k} (\mC^*)^k \mC^{j-k}   \qquad \textrm{for }j \in \N_0, \label{def-Uj}\\
g(x; t) &\coleq \langle x, \sum_{j=1}^\infty \frac{t^j}{j!} \mU_j x \rangle
= \|e^{-\bC t}x\|^2-\|x\|^2 \qquad \textrm{ for }t\ge 0 \textrm{ and }x \in \cH. \label{def-g}
\end{align}  

Note that, because $\|\mU_j\| \leq (2 \|\mC\|)^j$ for $j\in\N_0$, the series defining $\mQ(t)$ converges uniformly on bounded time intervals.

\makered{
To derive the upper bound for the propagator norm, we will use the following identity from \cite[Lemma A.2]{AAC2}:
\begin{lemma} \label{lm:sum-of-squares}
Let $\mU, \mV, \mW\in\cB(\cH)$. 
For all $m\in\N_0$, the following identity holds:
\begin{equation} \label{id:sum-of-squares}
 \begin{split}
&\sum_{j=1}^{\infty} \frac{t^j}{j!} \sum_{k=0}^{j-1} \tbinom{j-1}{k} \mU^k \mV \mW^{j-k-1}
\\
&= \sum_{j=0}^m \frac{t^{2j+1}}{(2j+1)!} \frac1{\binom{2j}{j}}
       \bigg( \sum_{k=0}^{\infty} \tfrac{(2j+1)!}{(k+2j+1)!} \tbinom{k+j}{j} t^k \mU^{k+j} \bigg) \mV
       \bigg( \sum_{\ell=0}^{\infty} \tfrac{(2j+1)!}{(\ell+2j+1)!} \tbinom{\ell+j}{j} t^\ell \mW^{\ell+j} \bigg) 
\\
&\quad +\sum_{j=2m+3}^{\infty} \frac{t^j}{j!} \sum_{k=m+1}^{j-m-2} \tbinom{j-1}{k} \Delta^{(m+1)}_{j,k} \mU^k \mV \mW^{j-k-1} \,,
 \end{split}
\end{equation}
where $\Delta^{(m)}_{j,k} \coleq \frac{\binom{k}{m} \binom{j-k-1}{m}}{\binom{k+m}{m} \binom{j-k-1+m}{m}}$ for all $m\leq \min(k, j-k-1)$.
\end{lemma}
}

We begin by proving the following upper bound for the propagator norm:
\begin{proposition}\label{prop:short-t-decay:upper}
Let the operator~$\mC\in\cB(\cH)$ be accretive with hypocoercivity index $\mHC\in\N_0$.
Then, there exist $t_2,c_2,C_2>0$ such that 
\begin{equation} \label{est:Propagator}
 \| e^{-\mC t} \|
\leq 1 -c_2 t^a  + C_2 t^{a+1}
\qquad \text{for } t \in [0, t_2],
\end{equation}
where $a =2 \mHC +1$. More precisely, we have
\begin{equation} \label{limsup:c2_tilde}
\limsup_{t \to 0} \frac{\norm{e^{-\bC t}} - 1 }{t^{2m_{HC}+1}} \le - \tilde c_2 
\end{equation}
with $\tilde c_2=c$ from \eqref{def-c}: 
\small
\begin{equation} \label{c2_tilde}
 \tilde c_2
\coleq
 \frac{\lim\limits_{\delta\to 0}\ \Big[ \inf \big\{\norm{\sqrt{\bC_H}\bC^{m_{HC}} x}^2 : \\ x\in\sphere,\ 
 \norm{\sqrt{\bC_H}\bC^j x}^2
\le \delta\ \forall j=0,\dotsc,m_{HC}-1 \big\} \Big]}{(2m_{HC}+1)!\, \binom{2m_{HC}}{m_{HC}}}\,. 
\end{equation}
\normalsize
\end{proposition}

\medskip
\begin{proof}
When $\delta$ decreases in the definition of $\tilde c_2$, the infima are taken over gradually smaller sets. Hence, these infima increase with $\delta\to0$, yet they are uniformly bounded by $\|\bC\|^{2\mHC+1}/[(2m_{HC}+1)!\, \binom{2m_{HC}}{m_{HC}}]$. So, $\tilde c_2<\infty$ is well-defined. 

Consider
$\| e^{-\mC t} \|^2 =\sup_{\|x\|=1} \ip{x}{\mQ(t)x}$ for $t \ge 0$ and recall that $\mU_0 =\mI$ and $\mU_1 =-2\mCH$.
If $\mHC=0$ then $\mC$ is coercive such that $\bC_H\geq\kappa\bI$ for some $\kappa>0$.
Thus, for all $x\in\cH$ with $\norm{x}=1$, we derive the estimate 
\begin{align*}
\langle x, \mQ(t) x \rangle &= \norm{x}^2 - 2 t \langle x, \bC_H x \rangle + \langle x, \sum_{j=2}^\infty \frac{t^j}{j!} \mU_j x \rangle \\
& \le 1 - 2 \kappa t + M_1 t^2 \qquad \forall t \in [0,1]
\end{align*}
for some constant $M_1 > 0$. Consequently, estimate~\eqref{est:Propagator} with $a=1$ and $c_2 = \kappa$ holds. 

Now let $\mHC \ge 1$, with $\kappa>0$ as in \eqref{Op:hypocoercive:kappa}.  For simplicity, we treat only the case $\mHC=1$; the generalization to operators with hypocoercivity index $\mHC\geq2$ is done similarly as in \cite[Lemma A.1]{AAC2} and hence is not  detailed here. To compute $\| e^{-\mC t} \|^2$ we consider $g(x;t)$ as in \eqref{def-g} for $t \ge 0$ and $x\in\sphere$.  
We define
\begin{equation} \label{lambda_x:mu_x:mHC_1}
 \lambda_x \coleq\ip{x}{\mCH x}\geq 0\ , \quad
 \mu_x \coleq\ip{x}{\mC^* \mCH \mC x}\geq 0\,,
\end{equation}
and note that $\lambda_x +\mu_x \geq \kappa >0$ for $x \in \sphere$. Fixing some $\delta \in (0, \kappa)$ for the moment, we first consider $x\in\sphere$ with $\lambda_x \leq \delta$. We use the identity (see \cite[(A.6)]{AAC2})
\begin{equation}\label{eq1}
\mU_j = (-1)^j 2 \sum_{k=0}^{j-1} \binom{j-1}{k} (\bC^*)^k \bC_H \bC^{j-1-k}, \qquad j \in \N
\end{equation}
and Lemma~\ref{lm:sum-of-squares} with $\mU=(-\mC)^*$, $\mV=\mU_1$, $\mW=-\mC$ and $m=0$ to derive
\begin{equation} \label{est:g:C0:tilde}
\begin{split}
 g(x;t)
 &= \big\langle x, \Big( \sum_{j=1}^{\infty} \frac{t^j}{j!}\ \mU_j \Big) x \big\rangle 
\\
 &= \sum_{j=1}^{\infty} \frac{t^j}{j!} \sum_{k=0}^{j-1} \tbinom{j-1}{k} \big\langle (-\mC)^k x, \mU_1 (-\mC)^{j-k-1} x  \big\rangle
\\
 &= t \big\langle \sum_{k=0}^{\infty} \frac{1}{(k+1)!} t^k (-\mC)^k x,  \mU_1
 \sum_{\ell=0}^{\infty} \frac{1}{(\ell+1)!} t^\ell (-\mC)^{\ell} x \big\rangle
\\
 &\quad +\frac{t^3}{3!} \tbinom{2}{1} \tfrac14 \big\langle (-\mC) x, \mU_1 (-\mC) x \big\rangle 
\\
 &\quad +\sum_{j=4}^{\infty} \frac{t^j}{j!} \sum_{k=1}^{j-2} \binom{j-1}{k} \Delta^{(1)}_{j,k} \big\langle (-\mC)^k x, \mU_1 (-\mC)^{j-k-1} x \big\rangle
\\
 &\leq \tfrac{1}{3!\, 2} \langle \mC x, \mU_1 \mC x \rangle \ t^3 +M_2 t^4 
\\
 &= -\tfrac{1}{3!} \mu_x t^3 +M_2 t^4 \qquad \forall t \in [0,1]
\end{split}
\end{equation}
for some constant $M_2>0$. 
Note that, in the estimate of \eqref{est:g:C0:tilde}, we only use $U_1\le0$ in the first term. 
Defining
\[ \mu_\delta \coleq \inf_{x \in \sphere, \lambda_x \le \delta} \mu_x = \inf_{x \in \sphere, \lambda_x \le \delta} \langle x, \bC^* \bC_H \bC x \rangle \ge \kappa - \delta > 0 ,
\]
we then have
\begin{equation}\label{est:g:C0:tilde:final}
 g(x;t)
 \leq -\frac{\mu_\delta}{3!} t^3 +M_2 t^4
 \qquad\text{ for all $x\in\sphere$ with $\lambda_x \leq\delta$ and $t \in [0,1]$}. 
\end{equation}

\makered{On the other hand, for fixed $\delta\in(0,\kappa)$, we will show next that there exists $t_\delta>0$ such that
\begin{equation} \label{est:g:C0:final}
 g(x;t) \leq - \frac{\mu_\delta}{3!} t^3 
\qquad \text{for all $x\in\sphere$ with $\lambda_x > \delta$ and $t\in[0,t_\delta]$}
\end{equation}
holds.
To prove~\eqref{est:g:C0:final} for $x\in\sphere$ with $\lambda_x > \delta$, we estimate $g(x;t)$ in~\eqref{def-g} as follows
\begin{equation} \label{est:g:C0}
 g(x;t) \le -2 \lambda_x t +M_3 t^2 \le - \lambda_x t \le -\frac{\mu_\delta}{3!} t^3\,,
\end{equation}
where the first inequality holds for all $t \in [0,1]$ and some $M_3>0$, the second for $t \le \lambda_x / M_3$, and the third for $t \le \sqrt{\lambda_x / d}$ with $d \coleq \mu_\delta/(3!)$. Hence, estimate~\eqref{est:g:C0:final} holds with $t_\delta \coleq\min\{1, \delta/{M_3}, \sqrt{ {\delta}/d}\}$.
}

Combining \eqref{est:g:C0:tilde:final} and \eqref{est:g:C0:final}, we obtain
\begin{equation}\label{a24}
g(x; t) \le - \frac{\mu_\delta}{3!}t^3 + M_2 t^4\quad \forall x \in \sphere\ \forall t \in [0, t_\delta]
\end{equation}
which gives
\begin{equation}\label{a5}
\frac{\norm{e^{-\bC t}}^2 -1 }{t^3} \le - \frac{\mu_\delta}{3!} + M_2 t\qquad \forall t \in [0, t_\delta]\quad \forall \delta \in (0, \kappa)
\end{equation}
and hence the upper bound in \eqref{est:Propagator} with the (non-sharp) constant $c_2 = \frac{\mu_\delta}{3!\, 2}$. Regarding the second claim, from \eqref{a5} we have that
\[ \limsup_{t \to 0} \frac{\norm{e^{-\bC t}}^2-1}{t^3} \le - \frac{\mu_\delta}{3!} \quad \forall \delta \in (0,\kappa) \]
and consequently
\[ \limsup_{t \to 0} \frac{\norm{e^{-\bC t}}^2-1}{t^3} \le - \frac{1}{3!} \lim_{\delta \to 0}\ \Big[ \inf \big\{ \norm{\sqrt{\bC_H}\bC x}^2 : \\
 x \in \sphere,\ 
\norm{\sqrt{\bC_H}x}^2 \le \delta \big\}\Big]\,. \]
Note that this limit exists because the given infimum is bounded from above by $\norm{\sqrt{\bC_H}\bC}^2$ and non-decreasing for $\delta \to 0$. Expanding the square root as in $\sqrt{1+\tau} = 1 + \tau/2 + \bigO(\tau^2)$ as $\tau \to 0$ we finally obtain
\[
\limsup_{t \to 0} \frac{\norm{e^{-\bC t}}-1}{t^3} 
\le - \frac{1}{3!\, 2} \lim_{\delta \to 0}\ \Big[ \inf \big\{ \norm{\sqrt{\bC_H}\bC x}^2 : \\
 x \in \sphere,\ 
\norm{\sqrt{\bC_H}x}^2 \le \delta \big\}\Big]\,. 
\]
This finishes the proof of~\eqref{limsup:c2_tilde} in the case $\mHC=1$.
\end{proof}


\subsection{Lower bound of the propagator norm}\label{sec:4-lower}

The following lemma is a refined version of \cite[Lemma A.4]{AAC2} that will be needed for establishing the lower bound for the propagator norm. 
From the proof of that lemma we know that
\begin{equation} \label{Delta_mjk}
  \Delta_{j,k}^{(\mHC)}\le1\qquad \mbox{for } 0\le\mHC\le k\le j-\mHC-1.
\end{equation}


\begin{lemma}\label{refined}
Let $\bC$ be hypocoercive with hypocoercivity index $m_{HC} \in \N$. Then there exist constants $b_\ell\in\R$, $\ell = 1, \dotsc ,m_{HC}$ depending only on $m_{HC}$ and $\tau_0  \in (0,\min(1,1 / \norm{\bC})]$ such that for each $\e>0$ and each $x_0 \in \bigcap_{j=0}^{m_{HC}-1} \{ x \in \cH:\ \norm{\sqrt{\bC_H}\bC^j x} \le \e\}$ the polynomial function $x_\tau\,:\,[0,1]\ni\tau\mapsto x_\tau\in\cH$, given by
\begin{equation}\label{starr}
x_\tau := x_0 + \sum_{\ell=1}^{m_{HC}} b_\ell \tau^\ell \bC^\ell x_0
\end{equation}
satisfies
\begin{equation}\label{gminusc}
\Abso{ g(x_\tau; \tau) + 2 c_1(x_0)\tau^{2m_{HC}+1}} \le  D_1 \tau \e^2   + D_2\tau^{2m_{HC}+2} \norm{x_0}^2
\end{equation}
for $\tau \in[0,\tau_0]$. Moreover, the constants $D_1, D_2>0$ depend only on $\bC$, and 
\begin{equation}\label{c1-def}
c_1(x_0) \coleq \frac{\norm{\sqrt{\bC_H} \bC^{m_{HC}} x_0}^2}{(2m_{HC}+1)!\, \binom{2m_{HC}}{m_{HC}}}. 
\end{equation}
\end{lemma}

Note that in the following proofs, we explicitly keep track of the dependence of constants on given data: Writing $C=C(x,y)$,  e.g., signifies that the constant $C$ depends \emph{only} on $x$ and $y$.

\begin{proof}
{\bf Step 1 - determining the constants $b_\ell$:}\\
For $x \in \cH$ and $\tau\in[0,1]$, consider $g(x;\tau)=\langle x,\sum_{j=1}^\infty \frac{\tau^j}{j!}\mU_j x\rangle$ using $\mU_j$ in the form \eqref{eq1}.
Employing Lemma \ref{lm:sum-of-squares} with $\mU = -\bC^*$, $\mV = \mU_1$, $\mW = -\bC$ and $m=m_{HC}-1$ we obtain
\begin{align}
\nonumber g(x; \tau) &
=\sum_{j=1}^{\infty} \frac{\tau^j}{j!}\sum_{k=0}^{j-1} \binom{j-1}{k} \langle(-\mC)^k x,\, \mU_1(-\mC)^{j-1-k}x\rangle  \\ 
&= \sum_{j=0}^{m_{HC}-1} \frac{\tau^{2j+1}}{(2j+1)!} \frac{1}{\binom{2j}{j}} \, \Big\langle \sum_{k=0}^\infty \frac{(2j+1)!}{(k+2j+1)!} \binom{k+j}{j}\tau^k (-\bC)^{k+j} x, \\
\nonumber &\quad\qquad \mU_1 \sum_{k=0}^{\infty} \frac{(2j+1)!}{(k+2j+1)!}\binom{k+j}{j} \tau^k (-\bC)^{k+j} x \Big\rangle \\
\nonumber &\quad + \frac{\tau^{2m_{HC}+1}}{(2m_{HC}+1)!} \binom{2m_{HC}}{m_{HC}} \Delta^{(m_{HC})}_{2m_{HC}+1, m_{HC}} \langle (-\bC)^{m_{HC}} x, \mU_1 (-\bC)^{m_{HC}} x \rangle \\
\nonumber &\quad + \sum_{j=2m_{HC}+2}^\infty \frac{\tau^j}{j!} \sum_{k=m_{HC}}^{j-m_{HC}-1} \binom{j-1}{k} \Delta^{(m_{HC})}_{j,k} \langle (-\bC)^k x, \mU_1 (-\bC)^{j-k-1} x \rangle \,.
\end{align}
Setting $x = x_\tau$, using $\mU_1 = - 2 \bC_H$ and
\[ \binom{2m_{HC}}{m_{HC}} \Delta^{(m_{HC})}_{2m_{HC}+1, m_{HC}} = \binom{2m_{HC}}{m_{HC}}^{-1} \]
the term to be estimated in \eqref{gminusc} equals
\begin{align}
\nonumber & g(x_\tau; \tau) + 2 c_1(x_0) \tau^{2m_{HC}+1} \\
& = - 2 \sum_{j=0}^{m_{HC}-1} \frac{\tau}{(2j+1)!} \frac{1}{\binom{2j}{j}} \Norm{ \sqrt{\bC_H} \sum_{k=0}^\infty \frac{(2j+1)!}{(k+2j+1)!} \binom{k+j}{j} \tau^{k+j} (-\bC)^{k+j} x_\tau}^2 \label{o1} \\
& \quad - 2 \frac{\tau^{2 m_{HC}+1}}{(2m_{HC}+1)!\, \binom{2m_{HC}}{m_{HC}} } \left( \norm{\sqrt{\bC_H} \bC^{m_{HC}} x_\tau}^2 - \norm{\sqrt{\bC_H} \bC^{m_{HC}} x_0}^2 \right) \label{o2} \\
& \quad + \tau^{2m_{HC}+2} \sum_{j=2m_{HC}+2}^{\infty} \frac{\tau^{j-2m_{HC} - 2}}{j!} \sum_{k=m_{HC}}^{j-m_{HC}-1} \binom{j-1}{k} \Delta^{(m_{HC})}_{j,k} \langle (-\bC)^k x_\tau, \mU_1 (-\bC)^{j-k-1} x_\tau \rangle. \label{o3}
\end{align}
Given the form of \eqref{starr} we set $b_0:=1$ and determine $b_\ell$ by examining the norm in \eqref{o1} for $j = m_{HC}-\ell$, $\ell = 1, 2, \dotsc, m_{HC}$, 
with the goal to annihilate the term containing $\tau^{m_{HC}}$, while the lower order terms will turn out to be of order $\varepsilon^2$.  Abbreviating
\[ c_{\ell, k} \coleq \frac{(2(m_{HC}-\ell)+1)!}{(k+2(m_{HC}-\ell)+1)!} \binom{k+m_{HC}-\ell}{m_{HC}-\ell} \ne 0 \]
we compute the norm in \eqref{o1} for fixed $\ell$ (in $\{1,...,\mHC\}$) by rearranging the double sum:
\begin{align}
& \Norm{\sqrt{\bC_H} \sum_{k=0}^\infty c_{\ell,k} \tau^{k+m_{HC}-\ell} (-\bC)^{k+m_{HC}-\ell} \sum_{p=0}^{m_{HC}} \tau^p b_p \bC^p x_0} 
\nonumber \\
&= \tau^{m_{HC}-\ell} \Norm{\sum_{k=0}^\infty \sum_{p=0}^{m_{HC}} \tau^{k+p} c_{\ell, k} b_p \sqrt{\bC_H} \bC^{k+m_{HC}-\ell+p} (-1)^{k+\mHC-\ell}x_0} 
\nonumber \\
&= \tau^{m_{HC}-\ell} \Norm{\sum_{i=0}^\infty \tau^i \sum_{r=0}^{\min(m_{HC}, i)} c_{\ell, i-r} b_r \sqrt{\bC_H} \bC^{i+m_{HC}-\ell} (-1)^{i-r+\mHC-\ell} x_0}.  \label{dritte}
\end{align}
In the last expression, the inner sum for fixed \fbox{$i = \ell$} equals
\begin{equation}\label{innensumme}
\sum_{r=0}^{\ell} (-1)^{\mHC-r} c_{\ell, \ell-r} b_r \sqrt{\bC_H}\bC^{m_{HC}} x_0.
\end{equation}
Now \eqref{innensumme} vanishes for each $\ell=1,2,\dotsc,m_{HC}$ by choosing the $b_r$, $r = 1, \dotsc, m_{HC}$ iteratively such that
\begin{equation}\label{b-eq}
 \sum_{r=0}^\ell (-1)^{m_{HC}-r} c_{\ell, \ell-r} b_r = 0,
\qquad \ell=1, \dotsc, m_{HC}. 
\end{equation}
From this construction it is clear that the constants $b_\ell$, $\ell = 0, \dotsc, m_{HC}$ depend only on $m_{HC}$. \\

{\bf Step  2 - estimate of the remaining terms in $g(x_\tau;\tau)$:} \\
Next we estimate the remaining terms of \eqref{dritte} for $\tau \in [0,1]$:
First, we consider those terms with \fbox{$0\leq i\leq\ell-1$}, which can be bounded from above by
\begin{multline} \label{k1}
\tau^{m_{HC}-\ell} \norm{\sum_{i=0}^{\ell-1} \tau^i \sum_{r=0}^i c_{\ell, i-r} b_r \sqrt{\bC_H} \bC^{i+m_{HC}-\ell} (-1)^{i-r+\mHC-\ell} x_0} 
\\ \leq
\sum_{i=0}^{\ell-1} \tau^{m_{HC}-\ell+i} \left( \sum_{r=0}^i c_{\ell, i-r} \abso{b_r}\right) \norm{\sqrt{\bC_H} \bC^{i+m_{HC} - \ell} x_0}
\leq
C_1 \e
\end{multline}
for some $C_1 = C_1(m_{HC})>0$ by using $\|\sqrt{\mC_H} \mC^j x_0\|\le \varepsilon$ for $j=0,...,\mHC-1$ in the final estimate.

Finally, we consider those terms in~\eqref{dritte} with \fbox{$i\geq\ell+1$}. Fix any $\tau_0 < \min(1, 1/\norm{\bC})$. For $\tau \le \tau_0  < \min(1,1 / \norm{\bC})$, using an index shift $k:=i-\ell-1$, we estimate
\begin{equation} \label{k2}
\begin{split}
& \tau^{m_{HC}-\ell} \norm{\sum_{i=\ell+1}^\infty \tau^i \sum_{r=0}^{\min(m_{HC}, i)} c_{\ell, i-r} b_r \sqrt{\bC_H} \bC^{i+m_{HC} - \ell} (-1)^{i-r+\mHC-\ell} x_0}
\\
& \leq
 \tau^{m_{HC}+1} \left( \sup_{k \in \N_0} \sum_{r=0}^{\min(m_{HC}, k+\ell+1)} c_{\ell, k+\ell+1-r} \abso{b_r} \right) \, \sum_{k=0}^\infty \tau^k \norm{\sqrt{\bC_H} \bC^{m_{HC}+1+k} x_0}
\\
& \leq
 \tau^{m_{HC}+1} C_2 \norm{x_0} ,
\end{split}
\end{equation}
for some $C_2 = C_2(m_{HC}, \norm{\bC})$.

Using~\eqref{b-eq}--\eqref{k2}, we estimate~\eqref{dritte} by
$C_1 \e + \tau^{m_{HC}+1} C_2 \norm{x_0}$ for $\tau \le \tau_0  < \min(1,1 / \norm{\bC})$, where the constants $C_1$ and $C_2$ depend only on $m_{HC}$ and $\norm{\bC}$. 
All in all, this means that (the absolute value of) \eqref{o1} can be estimated by
\[ C_3 \tau \e^2 + C_4 \tau^{2m_{HC}+3} \norm{x_0}^2 \qquad \textrm{ for } \tau \le \tau_0  < \min(1,1 / \norm{\bC}) \]
with constants $C_3 =C_3(m_{HC})$ and $C_4 = C_4(m_{HC}, \norm{\bC})$.

Using (with $b_0=1$)
\begin{align*} 
  \Big|\norm{\sqrt{\bC_H} \bC^{m_{HC}} x_\tau}^2 - & \norm{\sqrt{\bC_H}\bC^{m_{HC}} x_0}^2\Big| 
  \\
  & = \Big|\| \sqrt{\bC_H} \bC^{m_{HC}} \Big( \sum_{\ell=0}^{m_{HC}} b_\ell \tau^\ell \bC^\ell x_0\Big) \|^2 - \norm{\sqrt{\bC_H}\bC^{m_{HC}} x_0}^2\Big| \\
  & \le C_5 
  \tau \norm{x_0}^2 
\end{align*}
with $C_5 = C_5(m_{HC}, \norm{\bC})$, we can estimate \eqref{o2} by
\[ 
C_6 \tau^{2m_{HC}+2} \norm{x_0}^2 \] 
with $C_6 = C_6(m_{HC}, \norm{\bC}$.

Finally, to estimate~\eqref{o3}, we first note that for $\tau \in [0,1]$,
\begin{align*}  
\norm{x_\tau}^ 2 & \le \left( \sum_{\ell=0}^{m_{HC}} \abso{b_\ell}  \underbrace{\tau^\ell}_{\le 1} \norm{\bC}^\ell \norm{x_0} \right)^2 \\
& \le \norm{x_0}^ 2 \left( \max_{\ell = 0, \dotsc, m_{HC}} \abso{b_\ell}\right)^2 \left( \sum_{\ell=0}^{m_{HC}} \norm{\bC}^\ell \right)^2 \\
& \le C_7 \norm{x_0}^2
\end{align*}
for some constant $C_7 = C_7(m_{HC}, \norm{\bC})$. Using~\eqref{Delta_mjk}, the sum in expression~\eqref{o3} can be estimated as
\begin{gather*}
\Abso{ \sum_{j=2 m_{HC} + 2}^\infty \frac{\tau^{j-2m_{HC}-2}}{j!} \sum_{k=m_{HC}}^{j-m_{HC}-1} \binom{j-1}{k} \Delta^{(m_{HC})}_{j,k} \langle (-\bC)^k x_\tau, \mU_1 (-\bC)^{j-k-1} x_\tau \rangle } \\
\le \sum_{j=2m_{HC}+2}^\infty \frac{\tau^{j-2m_{HC}-2}}{j!} \sum_{k=m_{HC}}^{j-m_{HC}-1} \binom{j-1}{k} \norm{\bC}^{j-1}  \underbrace{\norm{\mU_1}}_{\le 2 \norm{\bC}} \norm{x_\tau}^2 \\
\le \sum_{j=2 m_{HC}+2}^\infty \frac{\tau^{j-2m_{HC}-2}}{j!} 2 \norm{\bC}^j C_7 \norm{x_0}^2 \underbrace{\sum_{k=m_{HC}}^{j-m_{HC}-1} \binom{j-1}{k}}_{\le 2^{j-1}}\,.
\end{gather*}
The last series converges uniformly for $\tau \in [0,1]$ due to the ratio test, as
\[
\frac{2 \tau \norm{\bC}}{j+1} < 1
\]
holds for such $\tau$ and large $j$. So we can estimate \eqref{o3}
by
\[ C_8 \tau^{2m_{HC} +2} \norm{x_0}^2 \]
uniformly for $\tau \in [0,1]$ and some constant $C_8 = C_8 (m_{HC}, \norm{\bC})$.

Putting everything together, we now have established
\begin{align*} \abso{g(x_\tau; \tau) + 2c_1(x_0) \tau^{2m_{HC}+1}} & \le C_3 \tau \e^2 + \left( C_4 \tau^{2m_{HC}+3} + C_6 \tau^{2m_{HC}+2} + C_8 \tau^{2m_{HC}+2} \right)\norm{x_0}^2 \\
& \le D_1 \tau \e^2 + D_2 \tau^{2m_{HC}+2} \norm{x_0}^2
\end{align*}
for $\tau \le \tau_0  < \min(1,1 / \norm{\bC})$ and constants $D_1 = D_1(m_{HC}), D_2  = D_2(m_{HC}, \norm{\bC})$.
\end{proof}

Note that the constants $b_\ell$ appearing in the above proof are exactly the constants which have been used in the finite dimensional case (see \cite[(A.51)]{AAC2}, which coincides with \eqref{b-eq} here). 
But in contrast to the analogous Lemma A.4 in \cite{AAC2}, the above proof is much more technical, since a number of additional non-vanishing terms had to be estimated here. The basic reason for that was explained in Remark \ref{remark_unterschied} \ref{remark_unterschied_3}.


We can now establish the desired lower bound (the finite dimensional analog is \cite[Lemma A.3]{AAC2}).

\begin{proposition}\label{lower}
Let the operator~$\mC\in\cB(\cH)$ be accretive with hypocoercivity index $\mHC\in\N_0$.
Then, there exist $t_1,c_1,C_1>0$ such that
\begin{equation}\label{starstar}
\norm{e^{-\bC t}} \ge 1 - c_1 t^a - C_1 t^{a+1} \qquad \textrm{for }t \in [0, t_1],
\end{equation}
where $a = 2m_{HC}+1$. More precisely, we have
\begin{equation}\label{liminf:c1_tilde}
\liminf_{t \to 0} \frac{\norm{e^{-\bC t}} - 1 }{t^{2m_{HC}+1}} \ge - \tilde c_1 
\end{equation}
with
\small
\begin{equation}\label{c1_tilde}
 \tilde c_1 
\coleq
 \frac{\inf\Bigl\{ \limsup\limits_{n\to\infty} \norm{\sqrt{\bC_H} \bC^{m_{HC}} x_n}^2 : \\ (x_n)_{n\in\N}\subset\sphere;\ \norm{\sqrt{\bC_H}\bC^j x_n}^2 \le \tfrac1n\ \forall j=0,\dotsc,m_{HC}-1\ \forall n \in \N  \Bigr\}}{(2m_{HC}+1)!\, \binom{2m_{HC}}{m_{HC}}} . 
\end{equation}
\normalsize
\end{proposition}

\begin{proof}
The (coercive) case $m_{HC}=0$ is easy to see:
 for $\norm{x}=1$
\begin{align*}
    \norm{e^{-\bC t} x}^2 &= \langle x, \mQ(t)x \rangle \\
    & = 1 + t \langle x, \mU_1 x \rangle + \langle x, \sum_{j=2}^\infty \frac{t^j}{j!} \mU_j x \rangle \\
    & = 1 - 2t \norm{\sqrt{\bC_H}x}^2 + t^2 \langle x, \sum_{j=0}^\infty \frac{t^j}{(j+2)!} \mU_{j+2} x \rangle \\
    & \ge 1 - 2t \norm{\sqrt{\bC_H}x}^2 -\tilde C_1t^2\quad \textrm{for } t \in [0,1]\,, 
\end{align*}
where we used $\|\mU_j\| \le (2\|\bC\|)^j$, $t_1=1$ and $\tilde C_1= \sum_{j=0}^\infty \frac{t^j}{(j+2)!} (2\norm{\bC})^j$.
Taking the supremum over all $x\in\sphere$ and forming the square root, we obtain \eqref{starstar} with $c_1 = \inf\{ \norm{\sqrt{\bC_H} x}^2 : x \in \sphere \} \ge \kappa$, with $\kappa$ as in \eqref{Op:hypocoercive:kappa}.

Suppose now that $m_{HC}\ge 1$.
We start with a sequence $(x_n)_{n\in\N}$, $x_n\in\sphere$ of unit vectors in~$\cH$, each satisfying
\begin{equation}\label{einsdurchn}
\norm{ \sqrt{\bC_H} \bC^j x_n } \le \frac{1}{n} \qquad \forall j=0,\dotsc,m_{HC}-1, 
\end{equation}
which exists by definition of the hypocoercivity index (cf.\ Remark \ref{remark_unterschied}\ref{remark_unterschied_3}). Using Lemma \ref{refined} we obtain $\tau_0,\,D_1,\,D_2>0$ and for each $n \in \N$ a function $(x_{n, \tau})_{\tau \in [0,\tau_0]}$ such that
\begin{equation}\label{g-below}
g(x_{n,\tau}; \tau) \ge -2 c_1(x_{n})\tau^{2m_{HC}+1} - D_1 \frac{\tau}{n^2} - D_2  \tau^{2m_{HC}+2} 
\,,\qquad \tau\in[0,\tau_0]. 
\end{equation}
From \eqref{Op:hypocoercive:kappa} and \eqref{einsdurchn} we have
$$
  \norm{ \sqrt{\bC_H} \bC^{\mHC} x_n }^2\ge\kappa-\sum_{j=0}^{\mHC-1} \norm{ \sqrt{\bC_H} \bC^j x_n }^2\ge\kappa-\frac{\mHC}{n^2} \ge \frac{\kappa}{2}\quad \mbox{for } n\ge n_0:=\sqrt{\frac{2\mHC}{\kappa}}.
$$
Hence, the constants $c_1(x_n)\ge0$, defined in \eqref{c1-def}, are uniformly bounded and bounded away from zero for $n\ge n_0$.

Let $\tilde x_{n, \tau} := \frac{x_{n,\tau}}{\norm{x_{n,\tau}}}$ and note that
\begin{equation}\label{xn-conv}
\norm{x_{n,\tau}} \to 1\quad \mbox{for } \tau\to 0,\quad \mbox{uniformly in }n\,,
    \end{equation}
because the constants $b_\ell$ in Lemma \ref{refined} do not depend on $x_n$ but only on the hypocoercivity index. 

For coupling $n$ with $\tau$ in $\tilde x_{n, \tau}$\,, we choose a function $h \colon [0,1] \to \N$ such that
\begin{equation}\label{h-conv}
\frac{1}{h(\tau)^2} \frac{1}{\tau^{2m_{HC}}} = \bigO(\tau) 
\qquad \textrm{for }\tau \to 0 \,,    
\end{equation}
which implies $\lim_{\tau \to 0} h(\tau) = \infty$ and $\lim_{\tau \to 0} \norm{x_{h(\tau), \tau}} = 1$. With  \eqref{g-below} and this function $h$ we obtain the lower bound
\begin{align}\label{low-bound}
\frac{\norm{e^{-\bC \tau}}^2 - 1}{\tau^{2m_{HC}+1}} &\ge \frac{\norm{e^{-\bC \tau}\tilde x_{h(\tau), \tau}}^2 -1 }{\tau^{2m_{HC}+1}} 
 = \frac{g(\tilde x_{h(\tau), \tau};\tau)}{\tau^{2m_{HC}+1}}  = \frac{g(x_{h(\tau), \tau};\tau)}{\tau^{2m_{HC}+1} \norm{x_{h(\tau), \tau}}^2}\\
&\ge - 2 \frac{c_1(x_{h(\tau)})}{\norm{x_{h(\tau), \tau}}^2} 
- D_1 \frac{1}{h(\tau)^2 \tau^{2m_{HC}} \norm{x_{h(\tau), \tau}}^2} - D_2 \frac{\tau}{\norm{x_{h(\tau), \tau}}^2} 
\,,\quad \tau\in[0,\tau_0]. \nonumber
\end{align}
Due to \eqref{xn-conv} and \eqref{h-conv}, $\frac1{\norm{x_{h(\tau), \tau}}}$ and $\frac{1}{h(\tau)^2} \frac{1}{\tau^{2m_{HC}}}$ are uniformly bounded on $[0,\tau_0]$. Moreover, the constants $c_1(x_{h(\tau)})$ are uniformly bounded away from zero for $h(\tau)\ge n_0$, or equivalently for $\tau\le t_1$ with some $t_1>0$.  Hence, \eqref{low-bound}  
gives \eqref{starstar}. 

In general it is not known if $\norm{e^{-\bC \tau}}$ is $(2m_{HC}+1)$-times differentiable at $\tau=0$. Hence we estimate now only ${\displaystyle  \liminf_{\tau\to 0}} \frac{\norm{e^{-\bC \tau}}^2 - 1 }{\tau^{2m_{HC}+1}}$ --- to show \eqref{liminf:c1_tilde}:
\[ \liminf_{\tau\to 0} \frac{\norm{e^{-\bC \tau}}^2 - 1 }{\tau^{2m_{HC}+1}} \ge - 2  \limsup_{n \to \infty}  c_1(x_n), \]
and we recall that $\displaystyle \limsup_{n \to \infty}  c_1(x_n)>0$. Taking the supremum over all sequences $(x_n)_{n\in\N}$ that satisfy  \eqref{einsdurchn}, we obtain
\begin{multline*}
\liminf_{\tau\to 0} \frac{\norm{e^{-\bC \tau}}^2 - 1 }{\tau^{2m_{HC}+1}} \ge - 2 
\inf \big\{ \limsup_{n \to \infty} c_1(x_n):\\
(x_n)_{n\in\N}\subset \sphere \textrm{ with } 
\norm{\sqrt{\bC_H}\bC^j x_n}^2 \le 1/n \textrm{ for }j=0,\dotsc,m_{HC}-1\textrm{ and }n \in \N \big\},
\end{multline*}
as claimed in \eqref{liminf:c1_tilde}. 
\end{proof}

In the next lemma we show that the leading order in the upper and lower bounds coincide. Moreover, in the finite-dimensional case they simplify to the bound obtained in \cite[Theorem 2.7]{AAC2}.
\begin{lemma} \label{lem:eqqq}
Let the operator~$\mC\in\cB(\cH)$ be accretive with HC-index~$\mHC\in\N_0$.
\begin{enumerate}[label=(\alph*)]
 \item \label{c1=c2}
Then the bounds in Propositions \ref{prop:short-t-decay:upper} and \ref{lower} agree in the leading order, i.e., we have $\tilde c_1=\tilde c_2$ with $\tilde c_1$ given in~\eqref{c1_tilde} and $\tilde c_2$ given in~\eqref{c2_tilde}, respectively.
 \item \label{c}
With $c= \tilde c_1= \tilde c_2$ as given in part~\ref{c1=c2}, we have
\begin{equation}\label{limit:c}
 \lim_{t \to 0} \frac{\norm{e^{-\bC t}} -1}{t^{2m_{HC}+1}} = -c , 
\end{equation}
hence \eqref{d25} follows.
\end{enumerate}
\end{lemma}

\begin{proof}
Part~\ref{c1=c2}:
Considering the definition of $\tilde c_2$, we can choose a (minimizing) sequence $(x_n)_{n\in\N}$ in $\sphere$ with $\norm{\sqrt{\bC_H}\bC^j x_n}^2 \le \frac{1}{n}$ for $j=0,\dotsc,m_{HC}-1$ such that
\[
 \tilde c_2 
= C \lim_{n \to \infty} \norm{\sqrt{\bC_H}\bC^{m_{HC}}x_n}^2 \\
= C \limsup_{n \to \infty} \norm{\sqrt{\bC_H}\bC^{m_{HC}}x_n}^2 
\ge  \tilde c_1
\]
holds with $C := \big[(2m_{HC}+1)!\, \binom{2m_{HC}}{m_{HC}}\big]^{-1}$. 

Conversely, let $(x_n)_{n\in\N} 
\subset \sphere$ 
with $\norm{\sqrt{\bC_H}\bC^j x_n}^2 \le \frac{1}{n}$ for $j=0,\dotsc,m_{HC}-1$ and $n \in \N$. We then have, for each $n \in \N$,
\begin{multline*}
C \inf \big\{ \norm{\sqrt{\bC_H}\bC^{m_{HC}} x}^2 :\  x \in \sphere,\ 
\norm{\sqrt{\bC_H}\bC^j x}^2 \le \frac{1}{n}\ \forall j=0,\dotsc,m_{HC}-1 \big\} \\
\le C \norm{\sqrt{\bC_H}\bC^{m_{HC}} x_n}^2\,.
\end{multline*}
By taking the $\displaystyle\limsup_{n\to\infty}$, 
which gives on the left-hand side the limit for $n \to \infty$ because of monotonicity (see the proof of Proposition \ref{prop:short-t-decay:upper} for more details), we find
\[ \tilde c_2 \le C \limsup_{n\to\infty} \norm{\sqrt{\bC_H}\bC^{m_{HC}}x_n}^2. \]
Taking next the infimum over all sequences $(x_n)_{n\in\N}$, $x_n\in \sphere$, we obtain $\tilde c_2 \le \tilde c_1$ and thus the stated equality. 

Part~\ref{c}:
Combining the estimates in~\eqref{limsup:c2_tilde} and~\eqref{liminf:c1_tilde} with the identity $\tilde c_1=\tilde c_2$ in part~\ref{c1=c2}, shows that the limit in~\eqref{limit:c} exists and is equal to $-c$.
This implies~\eqref{d25} and finishes the proof.
\end{proof}

We are now finally in the position to combine the above results to obtain the claimed short-time decay estimates.

\begin{proof}[Proof of Theorem \ref{thm:short-t-decay:sandwich}(a)]
The estimates are given in Propositions \ref{prop:short-t-decay:upper} and \ref{lower}, and Lemma~\ref{lem:eqqq}. We see $c_1 \ge c \ge c_2$ directly from inserting \eqref{d25} into \eqref{short-t-decay:sandwich}, leading to
\[ -c_1 - C_1 t \le - c + h(t) \le - c_2 + C_2 t \] 
for small $t$ with $\lim_{t \to 0} h(t) = 0$, and letting $t \to 0$.
\end{proof}

\subsection{Reverse direction of Theorem \ref{thm:short-t-decay:sandwich}}\label{sec:reverse}

\begin{lemma}\label{lem:fin-index}
    Let the operator $\bC\in\cB(\cH)$ be accretive and hypocoercive, or (equivalently due to Proposition \ref{prop:UniExpStableSG}) let the corresponding semigroup satisfy
    \begin{equation}\label{norm-smaller1}
        \|\bP(t_0)\|<1\quad \mbox{for some } t_0>0\ .
    \end{equation}
    Then, $\bC$ has a finite hypocoercivity index $\mHC\in\N_0$.
\end{lemma}

\begin{proof}
Set $p:=\|\mP(t_0)\|^2$ such that $p<1$ due to~\eqref{norm-smaller1}. For each $x\in\sphere$ we have: $\|\mP(t)x\|^2$ is analytic on $\R_0^+$ with
\[
  \|\mP(t)x\|^2 = \sum_{j=0}^\infty \langle x,\mU_j x\rangle \frac{t^j}{j!}\ .
\]
Using $\|\mU_j\| \le (2\|\bC\|)^j$, this series converges absolutely and uniformly w.r.t.\ $x\in\sphere$ and w.r.t.\ $t\in[0,t_0]$. Hence, there exists some (large enough) \emph{odd} integer $a$ such that
\begin{equation}\label{sum-est}
    \sum_{j=a+1}^\infty \frac{(2\|\bC\|t)^j}{j!} \le \frac{1-p}{2},\quad t\in[0,t_0]\ .
\end{equation}

We prove the statement in Lemma~\ref{lem:fin-index} by contradiction: Assume the uniform estimate~\eqref{Op:hypocoercive:kappa} is false for $m:=\frac{a-1}{2}\in\N_0$. Then, due to \eqref{HC_failure}, there exists a sequence
\begin{equation}\label{small-xn}
    (x_n)_{n\in\N}, \ x_n\in \sphere:\:\quad \|\sqrt{\bC_H}\bC^j x_n\|\le\frac1n\quad \forall\,j=0,...,m\ .
\end{equation}
Using \eqref{sum-est} we have:
\begin{eqnarray}\label{P-lower}
    \|\mP(t)x_n\|^2 &=&
    1+\displaystyle\sum_{j=1}^a \langle x_n,\mU_j x_n\rangle \frac{t^j}{j!}
    + \sum_{j=a+1}^\infty \langle x_n,\mU_j x_n\rangle \frac{t^j}{j!} \nonumber \\
    &\ge& 1 - \displaystyle\sum_{j=1}^a \big|\langle x_n,\mU_j x_n\rangle \big| \frac{t^j}{j!} -\frac{1-p}{2} \nonumber \\
    &=& \frac{1+p}{2} - \displaystyle\sum_{j=1}^a \big|\langle x_n,\mU_j x_n\rangle\big| \frac{t^j}{j!}\,,\quad t\in[0,t_0] \,,
\end{eqnarray}
where $p<\frac{1+p}{2}$ since $p<1$. To estimate the last sum, we use \eqref{eq1} with $j=1,...,2m+1$: 
$\langle x_n,\mU_j x_n\rangle$ is a linear combination of terms of the form $\langle \sqrt{\bC_H} \bC^k x_n,\sqrt{\bC_H} \bC^{j-1-k} x_n\rangle$ with $k\le m$ or $j-1-k\le m$. So, due to \eqref{small-xn}, the norm of one side of this inner product is bounded by $\frac1n$, and hence the inner product is bounded by
$$
  \frac1n \max\big(1, \|\bC\|^{j-\frac12}\big) =: \frac1n \alpha_j\,.
$$
Using \eqref{eq1} and $\sum_{k=0}^{j-1} \binom{j-1}{k} =2^{j-1}$ we obtain
\begin{equation}\label{uj-est}
    \big|\langle x_n,\mU_j x_n\rangle\big| \le \frac1n 2^j \alpha_j\,,\quad j=1,...,a\,.
\end{equation}
Using \eqref{uj-est} for \eqref{P-lower} yields
\[
  \|\mP(t)x_n\|^2 \ge \frac{1+p}{2} - \frac1n\sum_{j=1}^a \alpha_j\frac{(2t)^j}{j!} \quad \mbox{for } 0\le t\le t_0\,.
\]
For $t=t_0$ and $n$ large enough this contradicts \eqref{norm-smaller1}, since $p=\|\mP(t_0)\|^2<1$.

Thus, \eqref{Op:hypocoercive:kappa} holds for $m=\frac{a-1}{2}$ and the hypocoercivity index of $\bC$ satisfies $\N_0\ni \mHC\le m$.
\end{proof}

\begin{proof}[Proof of Theorem \ref{thm:short-t-decay:sandwich}(b)]
Because of \eqref{d25}, there exists some $t_0>0$ such that $\|\bP(t_0)\|<1$. So, Lemma \ref{lem:fin-index} shows that $\bC$ has a hypocoercivity index $\mHC\in\N_0$. But then, Theorem \ref{thm:short-t-decay:sandwich}(a) implies
$$
  \|\bP(t)\|=1-\hat c t^{2\mHC+1}+o(t^{2\mHC+1})\,.
$$
Thus $a=2\mHC+1\in2\N-1$.
\end{proof}

In this section we have established 
estimates for the short-time decay of the propagator norm in terms of the hypocoercivity index (if it exists). The next section provides a staircase form which allows to check the hypocoercivity properties.


\section{Staircase form}\label{sec_staircase}
In the finite-dimensional setting, by choosing appropriate bases for the associated spaces one can derive a so-called ``staircase form'' from which the property of hypocoercivity and the hypocoercivity index can be directly read off, see \cite{AAM21,AAM22}. There are also coordinate-free versions of such staircase forms (for the analysis of differential algebraic equations) constructed via sequences of subspaces, so called Wong sequences, see e.g. \cite{BerR13}. In this section we follow this path and derive a ``staircase-form'' that displays the hypocoercivity properties.

Given a direct sum decomposition $\cH = \bigoplus_{i=1}^s \cH_i$ ($s \in \N)$ with subspaces $\cH_i \subseteq \cH$ and an operator $\bT \in \cB(\cH)$ we denote by $\bT_{i,j} \in \cB(\cH)$ ($i,j = 1, \dotsc, s$) the $(i,j)$-th component of $\bT$ given by
\[ \bT_{i,j} = \pi_{\cH_i} \circ \bT \circ \iota_{\cH_j} \]
where $\pi_{\cH_i} \colon \cH \to \cH_i$ and $\iota_{\cH_i} \colon \cH_i \to \cH$ are the canonical projections and injections, respectively. In this case, $\bT$ can be represented by the $s \times s$ operator matrix $(\bT_{i,j})_{i,j = 1, \dotsc, s}$. For the following Lemma, if the spaces $\cH_i$ carry a superindex this will be reflected in the operators as well, i.e., for $\cH = \bigoplus_i \cH_i^s$ we write $\bT = (\bT^s_{i,j})_{i,j}$. Note that $(\bT_{i,j})^* = (\bT^*)_{j,i}$. We fix the notational ambiguity that occurs when leaving out the parentheses by setting $\bT_{i,j}^* \coleq (\bT_{i,j})^*$.

We call a self-adjoint or skew-adjoint operator in $\cB(\cH)$ \emph{nonsingular} if it is injective and has dense image (note that because of self-/skew-adjointness and equality of domain and codomain these two conditions are actually equivalent) and \emph{singular} otherwise.

The next result is called the staircase form for ($\bJ$, $\bR$), see \cite[Lemma 1]{AAM21}.

\begin{lemma}
Let $\bJ \in \cB(\cH)$ be skew-adjoint and $\bR \in \cB(\cH)$ self-adjoint with equal domains and $\dim \ker \bR < \infty$. Then there exists a direct sum decomposition $\cH = \bigoplus_{i=1}^s \cH_i$ with $s \ge 2$ and each $\cH_i$ a (possibly zero) subspace of $\cH$ such that:
\begin{itemize}
\item $\dim \cH_i < \infty$ for $i=2,\dotsc, s$;
\item with respect to this decomposition, $\bJ$ is an operator matrix of the form
\begin{equation}
\label{dienstag2}\bJ = 
\begin{bNiceArray}{ccccccc|c}[margin,nullify-dots]
\bJ_{1,1} & -\bJ_{2,1}^* & 0 & & \cdots & & 0 & 0 \\
\bJ_{2,1} & \bJ_{2,2} & -\bJ_{3,2}^* & & & & & \\
0 & \ddots & \ddots & \ddots & & & \vdots &  \\
& & \bJ_{k, k-1} & \bJ_{k, k} & -\bJ_{k+1, k}^ * & & & \vdots \\
\vdots & & & \ddots & \ddots & \ddots & 0 &  \\
& & & & \bJ_{s-2, 2-3} & \bJ_{s-2, s-2} & - \bJ_{s-1, s-2}^* & \\
0 & & \cdots & & 0 & \bJ_{s-1, s-2} & \bJ_{s-1, s-1} & 0 \\
\hline
0 & & & \cdots & & & 0 & \bJ_{s,s} 
\end{bNiceArray}
\end{equation}
i.e., $\bJ_{j,k} = - \bJ_{k,j}^*$, $\bJ_{j,k} = 0$ for $\abso{j-k} \ge 2$, $\bJ_{s, s-1} = \bJ_{s-1, s} = 0$;
\item if $s \ge 3$, $\bJ_{i, i-1}$ is surjective (and hence $-\bJ_{i, i-1}^*$ is injective) for $i = 2, \dotsc, s-1$; \item $\bR$ has the form
\begin{equation}
\bR =
\begin{bNiceArray}{cc}[margin,nullify-dots]
\bR_{1,1} & 0 \\
0 & 0
\end{bNiceArray}
\end{equation}
i.e., $\bR_{j,k} = 0$ for $j+k>2$ and $\bR_{1,1} = \bR_{1,1}^*$;
\item if $\bR$ is nonsingular then $s=2$ and $\cH_2 = \{0\}$.
\end{itemize}
\end{lemma}

\begin{proof}
We first recall  the identities $\ker \bT = (\im \bT^*)^\perp$, $(\ker \bT)^\perp = \overline{\im \bT^*}$ for any $\bT \in \cB(\cH)$. Then, we view $\bR$ as an operator $\bR \colon (\ker \bR)^\perp \oplus \ker \bR \to \overline{\im \bR} \oplus (\im \bR)^\perp$
and set
\begin{align*}
\cH^1_1 & \coleq (\ker \bR)^\perp = \overline{\im \bR}, \\
\cH^1_2 & \coleq \ker \bR = (\im \bR)^\perp
\end{align*}
to write $\bR, \bJ \in \cB(\cH^1_1 \oplus \cH^1_2)$ in components as follows:
\[ \bR = \begin{bmatrix}
\bR_{1,1} & 0 \\
0 & 0
\end{bmatrix}, \qquad
\bJ = \begin{bmatrix}
\bJ^1_{1,1} & \bJ^1_{1,2} \\
\bJ^1_{2,1} & \bJ^1_{2,2}
\end{bmatrix}.
\]
Note that due to $\iota_i^* = \pi_i$, $\bR_{1,1}^* = \bR_{1,1}$ and $(\bJ^1_{i,j})^* = - \bJ^1_{j,i}$. 
Moreover, $\dim \cH^1_2 < \infty$.
To proceed inductively, we now take note of the decomposition (starting with $i=1$)
\begin{gather}\label{dienstag1}
\cH = \bigoplus_{j=1}^{i+1} \cH^i_j\textrm{ with }\cH^i_2,\dotsc,\cH^i_{i+1}\textrm{ finite-dimensional},\\
\label{dienstag3}\bJ^i_{j,k} = 0 \textrm{ for }j,k=1, \dotsc, s\textrm{ with }\abso{j-k} \ge 2\,,\\
\label{dienstag4} \textrm{if }i \ge 2,\ \bJ^i_{j, j-1} \textrm{ is surjective for }j=2, \dotsc, i\,.
\end{gather}
Conditions \eqref{dienstag3} and \eqref{dienstag4} are void for $i=1$ but will be established for the next inductive steps.

In operator matrix form, $\bJ$ is given by
\[
\bJ =
\begin{bNiceArray}{ccccc}[margin,nullify-dots]
\bJ^i_{1,1} & -(\bJ^i_{2,1})^* & 0 & \cdots & 0 \\
\bJ^i_{2,1} & \bJ^i_{2,2} & -(\bJ^i_{3,2})^* & & \vdots \\
& \ddots & \ddots & \ddots & \\
\vdots & & \bJ^i_{i, i-1} & \bJ^i_{i, i} & -(\bJ^i_{i+1, i})^* \\
0 & \cdots & 0 & \bJ^i_{i+1, i} & \bJ^i_{i+1, i+1}
\end{bNiceArray}.
\]
If $\bJ^i_{i+1,i} = 0$ then the claim is established with $s = i+1$. If $\bJ^i_{i+1, i} \ne 0$ we obtain a decomposition of the form \eqref{dienstag1}--\eqref{dienstag4} with the upper index $i$ being replaced by $i+1$ as follows. 
Decomposing the codomain of $\bJ^i_{i+1,i} \colon \cH^i_i \to \cH^i_{i+1}$ as $\cH^i_{i+1} = \im \bJ^i_{i+1,i} \oplus (\im \bJ^i_{i+1,i})^\perp$ we set
\begin{align*}
\cH^{i+1}_j &\coleq \cH^i_j\textrm{ for }\quad j=1, \dotsc, i\,,  \\
\cH^{i+1}_{i+1} &\coleq \im \bJ^i_{i+1,i} \subset \cH_{i+1}^i\,, 
\\
\cH^{i+1}_{i+2} &\coleq (\im \bJ^i_{i+1,i})^\perp\subset \cH_{i+1}^i\,.
\end{align*}
With $\cH = \bigoplus_{j=1}^{i+2} \cH^{i+1}_j$ we now have $\bJ^{i+1}_{j,k} = \bJ^i_{j,k}$ for $j,k = 1, \dotsc, i$. It is clear that $\cH^{i+1}_2,\dotsc,\cH^{i+1}_{i+2}$ are finite-dimensional. We verify that $\bJ^{i+1}_{j,k} = 0$ for $\abso{j-k} \ge 2$: This is inherited for $j, k = 1, \dotsc, i$ and by skew-adjointness, it suffices to verify it for $j = i+1$ and $k = 1, \dotsc, i-1$ as well as for $j = i+2$ and $k = 1, \dotsc, i$. In the first case, we see that
\begin{align*}
\bJ^{i+1}_{i+1, k} &= \pi_{\cH^{i+1}_{i+1}} \circ \bJ \circ \iota_{\cH^{i+1}_k} = \pi_{\cH^{i+1}_{i+1}} \circ \pi_{\cH^i_{i+1}} \circ \bJ \circ \iota_{\cH^i_k} = \pi_{\cH^{i+1}_{i+1}} \circ \bJ^i_{i+1, k} = 0,
\end{align*}
because $\abso{i+1-k} \ge 2$. In the second case, we similarly have
\begin{align*}
\bJ^{i+1}_{i+2, k} &= \pi_{\cH^{i+1}_{i+2}} \circ \bJ \circ \iota_{\cH^{i+1}_k} = \pi_{\cH^{i+1}_{i+2}} \circ \pi_{\cH^i_{i+1}} \circ \bJ \circ \iota_{\cH^i_k} = \pi_{\cH^{i+1}_{i+2}} \circ \bJ^i_{i+1,k} = 0
\end{align*}
for $k = 1, \dotsc, i-1$ because then $\abso{i+1-k} \ge 2$ and also for $k=i$ because $\cH^{i+1}_{i+2} = (\im \bJ^i_{i+1, i})^\perp$. That $\bJ^{i+1}_{j, j-1}$ is surjective is inherited for $j = 2, \dotsc, i$, the case for $\bJ^{i+1}_{i+1, i}$ is seen from
\begin{align*}
\bJ^{i+1}_{i+1, i} = \pi_{\cH^{i+1}_{i+1}} \circ \bJ \circ \iota_{\cH^{i+1}_i} = \pi_{\cH^{i+1}_{i+1}} \circ \pi_{\cH^i_{i+1}} \circ \bJ \circ \iota_{\cH^i_i} = \pi_{\cH^{i+1}_{i+1}} \circ \bJ^i_{i+1, i}.
\end{align*}

Hence, we again are in the situation of \eqref{dienstag1}--\eqref{dienstag4} with $i$ replaced by $i+1$. This procedure is iterated until $\bJ^i_{i+1,i}=0$ holds, but we still
need to ascertain that this is the case after finitely many steps. 
To this end we note that, if $\bJ^i_{i+1, i} \ne 0$, then $1 \le \dim (\im \bJ^i_{i+1, i}) = \dim \cH^{i+1}_{i+1} \le \dim \cH^i_{i+1}$ and hence $\dim \cH^{i+1}_{i+2} = \dim \cH^i_{i+1} - \dim \cH^{i+1}_{i+1} < \dim \cH^i_{i+1}$. Consequently, if $\bJ^i_{i+1, i}$ does not become zero earlier, we will reach $\cH^{i+1}_{i+2} = \{0\}$ for $i = \dim \ker \bR$ 
and hence $\bJ^{i+1}_{i+2, i+1} = 0$. This implies $s \le \dim \ker \bR + 2$.

It is evident that, if $\bR$ is nonsingular, then $\cH^1_2 = \{0\}$ and $s=2$.
\end{proof}

\begin{remark}
If a nontrivial $\bJ_{s,s}$ is present in the staircase form, i.e., if $\cH_s \ne \{0\}$, then the operator $\bC = \bR - \bJ$ cannot be hypocoercive because $\bJ_{s,s}$ has only imaginary eigenvalues and $\cH_s$ is invariant under $\bJ_{s,s}$. Contrarily to the finite-dimensional setting, the staircase form cannot be used to determine whether the system in question satisfies \ref{B:Tmop} or the weaker condition \ref{p4'}.
So, in general, it does not determine the hypocoercivity index of $\bC$ (if it exists).
\end{remark}


\section{Example: Lorentz kinetic equation}\label{sec:lorentz}

In this section we consider the \emph{Lorentz kinetic equation}
\begin{equation}\label{lorentz}
 \partial_t f +\bv\cdot\nabla_\bx f = \sigma\left(\frac{1}{\meas(\sphere^{d-1})} \int_{\sphere^{d-1}} f\,\ud \bv - f \right)=:  \sigma (\tilde f - f)
\end{equation}
for the phase space distribution $f(\bx,\bv,t)$ with $(\bx,\bv,t) \in \torus^d \times \sphere^{d-1} \times \R^+$, some relaxation rate $\sigma>0$, and space dimension $d\ge2$. Here, $\tilde f$
denotes the mean of $f$ w.r.t.\ the velocity sphere $\sphere^{d-1}$.

This equation is a linear Boltzmann equation with collision operator $\cC f \coleq \sigma (\tilde f - f)$, which is local in position $\bx$. It describes the evolution of free particles (i.e., without external force) moving on the $d$-dimensional torus $\torus^d$ with speed $1$ (since the particle collisions preserve the kinetic energy and hence $|\bv|$). Originally, it was introduced to model a 3D-electron gas in metals \cite{lorentz1905}, but also its 2D-analog has been used to simulate phase transitions \cite{LoMa95}. 
It is well-known  (see e.g.~\cite[Theorem 3.1]{Golse}, \cite{UPG,hankwanleautaud}) that this equation exhibits exponential convergence to an equilibrium \makered{(but the decay rate is not quantified there)}:
\begin{theorem}\label{lorentz-HC}
The operator $\bA \coleq - \bv \cdot \nabla_{\bx} + \cC$ generates a $C_0$-semigroup on $\cH:=L^2(\torus^d \times \sphere^{d-1})$ and there exist constants $C,\lambda>0$ such that for any initial datum $f_0 \in \cH$ 
the solution of \eqref{lorentz} satisfies
\[ \left\| f(t) - f_\infty \right\|_{\cH} \le C e^{-\lambda t} \left\| f_0 \right\|_{\cH}, \]
where $f_\infty:=\frac{1}{(2\pi)^d\meas(\sphere^{d-1})} \int_{\torus^d \times \sphere^{d-1}} f_0(\bx,\bv) \,\ud \bv \,\ud \bx$ denotes the unique equilibrium corresponding to $f_0$.
\end{theorem}
In other words, this theorem states that $\bA$ is hypocoercive in the sense of Definition \ref{def:hypoco} on $\widetilde \cH = (\ker \bA)^\perp \subseteq \cH$ with the subspace topology. To see this, note that 
$\bA$ is an unbounded operator on $\cH$ with domain
\[ \cD(\bA) = \{ f \in \cH\ |\ \bv \cdot \nabla_{\bx} f \in \cH\}\,. \]
Following \S\ref{sec:HC}, it can be decomposed as
\[ -\bA = \bC = \bR - \bJ \]
with bounded self-adjoint part $\bR \in \cB(\cH)$ given by
\[ \bR(f)(\bx,\bv) = \sigma \Big(f(\bx,\bv) -  \frac{1}{\meas(\sphere^{d-1})} \int_{\sphere^{d-1}} f(\bx,\bw)\,\ud \bw \Big) \ge 0,
\]
and unbounded skew-adjoint part
\[ \bJ(f)(\bx,\bv) = -\bv \cdot \nabla_\bx f(\bx,\bv) \]
where $\cD(\bJ) = \cD(\bA)$. This implies that
\[ \ker \bA = \ker \bR \cap \ker \bJ = \Span\{1\}, \]
i.e., the kernel of $\bA$ consists of all constant functions on $\torus^d \times \sphere^{d-1}$. 
These are the \emph{global equilibria} of \eqref{lorentz}, while the $\bx$-dependent, but constant-in-$\bv$ functions of $\ker \bR=L^2(\torus^d_{\bx})\times \Span\{1\}$ form the \emph{local equilibria} of \eqref{lorentz}.

\makered{We will now analyze \eqref{lorentz} in $d=2$, setting $\sigma=1$ for simplicity. By constructing a Lyapunov functional, we will derive an explicit, quantitative decay rate and we will determine the hypocoercivity index of \eqref{lorentz} and thus its short-time behavior. }
Because $\bA$ is an unbounded operator, our theory developed so far does not apply directly. 
\makered{Therefore, we proceed in two steps: First, we make a mode by mode analysis of~\eqref{lorentz}. Then, this can be combined into results for the infinite dimensional system~\eqref{lorentz}. We note that the analysis of this example should mostly serve to present a general method that is applicable to a wide class of semi-dissipative hypocoercive problems, e.g.\ to the Goldstein-Taylor system in~\cite{AESW2021}.}

\subsection{\makered{Mode by mode analysis}}\label{mode-ana}

\makered{Similarly to \cite{AAM22Oseen,AAC16, BDMMS20} we} 
employ a Fourier decomposition on $\torus^2$, i.e.,
\[ f(\bx,\bv,t) = \frac{1}{2\pi}  \sum_{\bn \in \Z^2} f_\bn(\bv,t) e^{i \bn\cdot \bx}\textrm{ with } f_\bn(\bv,t) =  \frac{1}{2\pi} \langle f(\bx,\bv,t), e^{i \bn \cdot \bx} \rangle \textrm{ for }\bn \in \Z^2,
\]
to study the individual Fourier modes $f_\bn(\bv,t)$ of $f(\bx,\bv,t)$. This leads to the family of equations
\[ \partial_t f_\bn + i \bv \cdot \bn f_\bn = \cC f_\bn,\quad \bn \in \Z^2,\quad t>0, \]
or
\begin{equation} \label{lorentz_mode_n} 
 \partial_t f_\bn = - (\bR - \bJ_\bn) f_\bn 
\end{equation}
with $\bR, \bJ_\bn \in \cB(L^2(\sphere^{1}))$ given by $\bR(\phi) = \phi - \tilde \phi$ and $(\bJ_\bn \phi)(\bv) = -i (\bv \cdot \bn) \phi(\bv)$ for $\phi \in L^2(\sphere^{1})$. Clearly, $\bR$ is self-adjoint and $\bJ_\bn$ is skew-adjoint. 
With reference to Remark \ref{remark25}\ref{remark25.3} we note that the operator $\phi\mapsto \tilde\phi$ is compact on $L^2(\sphere^1)$, and hence $\bR$ is not. \makered{In the sequel we use the notation $\bC_\bn:=\bR-\bJ_\bn$, and the corresponding propagator is $\bP_\bn(t) :=e^{-\bC_\bn t}$.}

For $\bn=0$, $\bR - \bJ_0 = \bR$ is the projection onto $(\ker \bR)^\perp$, and its one-dimensional kernel $\ker \bR=\Span\{ 1 \}$ consists of all constant functions on $\sphere^1$. Hence, $\bR$ is coercive on $(\ker \bR)^\perp$ with $\kappa=1$.

For $\bn \ne 0$ we see that $\ker (\bR - \bJ_\bn) = \ker \bR \cap \ker \bJ_\bn = \{0\}$. 
Since $\bR$ is not coercive on $L^2(\sphere^{1})$, the hypocoercivity index of $\bR - \bJ_\bn$ (if it exists) is positive.

The following lemma states that \makered{each} operator $\bR - \bJ_\bn$ in fact is hypocoercive \makered{and that they all satisfy uniform (in $\bn$) decay estimates for short and large times. We note already in advance that such \emph{uniform} estimates are challenging for two reasons: On the one hand, $\|\bP_\bn(t)\|$ is typically not smooth on $[0,\infty)$ but only Lipschitz continuous (even for matrix exponentials). On the other hand, for hypocoercive modal systems like \eqref{lorentz_mode_n}, the skew-adjoint operator typically scales like the Euclidian norm of~$\bn$ which we denote by~$|\bn|$. Hence, the hypocoercive mixing increases with $|\bn|$ and the sequence $\|\bP_\bn(t)\|$ ``tends'' to decrease with increasing~$|\bn|$, but not pointwise in~$t$,  due to the non-smoothness of $\|\bP_\bn(t)\|$. Numerically, this can be verified easily for \eqref{lorentz_mode_n}.}

\begin{lemma}\label{lem1}
Let $d=2$. For each \makered{mode} $\bn = (n_1, n_2) \in \Z^2 \setminus\{0\}$ \makered{the following holds:
\begin{enumerate}
    \item[(a)] The hypocoercivity index is $m_{HC}(\bR - \bJ_\bn) = 1$. This holds uniformly in $\bn$ in the sense that there is a constant $\kappa>0$ (independent of $\bn$) such that for all $\bn \in \Z^2 \setminus\{0\}$ we have
    \[ \bR + \bJ_\bn \bR \bJ_\bn^* \ge \kappa \bI. \]
    \item[(b)] The norm of the solution to~\eqref{lorentz_mode_n} decays exponentially for long times like
    \begin{equation}\label{unif-decay}
        \|f_\bn(t)\|_{L^2(\sphere^{1})}
        \le \min\Big[1,\sqrt{\tfrac{2|\bn|+1}{2|\bn|-1}} e^{-\lambda_0 t} \Big] \|f_\bn(0)\|_{L^2(\sphere^{1})}    
        \le \sqrt3 e^{-\lambda_0 t}\|f_\bn(0)\|_{L^2(\sphere^{1})} 
    \end{equation}
    for $t\ge0$ with the (non-sharp) rate $\lambda_0=\frac12-\frac{1}{6\sqrt2}-\frac13\sqrt{\frac{7}{16}+\frac{1}{\sqrt8}}\approx0.0857$.
    \item[(c)] The propagator norm decays (algebraically) for short times like
\begin{equation}\label{Pn-decay}
    \|\bP_\bn(t)\|_{\cB(L^2(\sphere^1))} \le 1-c t^3,\quad 0\le t\le \tau,
\end{equation}
and the $\bn$-independent constants $c,\,\tau>0$ are given explicitly in the proof. 
\end{enumerate} }
\end{lemma}

\begin{proof}
\noindent
\makered{\underline{Part (a):}} According to Lemma \ref{lem:Op-Equivalence} \ref{B:Tmop} we need to show that
\begin{equation}\label{est1}
\bR + \bJ_\bn \bR \bJ_\bn^* \ge \kappa \bI
\end{equation}
for some $\kappa>0.$
Since the operator $\bR$ is invariant under rotations of $\sphere^1$ and $\bJ_{(n_1, 0)}$ is a multiple of $\bJ_{(1,0)}$, it suffices to prove the result for $\bJ_{(1,0)}$ in order to obtain the coercivity estimate \eqref{est1} uniformly in all $\bn \ne 0$.

The proof of \eqref{est1} in $\cB(L^2(\sphere^1))$ is based on a basis representation of $\bR$ and $\bJ_{(1,0)} \bR \bJ_{(1,0)}^* = v_1 \circ \bR \circ v_1$, with $\bv(\varphi) = (\cos\varphi, \sin\varphi)$ and $v_1 = \cos \varphi$. Using the orthonormal basis
\[ \phi_j(\varphi) = \frac{1}{\sqrt{2\pi}} e^{ij\varphi},\qquad j \in \Z,\: \varphi \in [0,2\pi]\, \]
of $L^2(\sphere^1)$, $\bR$ is given by $\bR (\phi_j) = (1 - \delta^j_0)\phi_j$. 
For $v_1 \circ \bR \circ v_1$, a short calculation gives
\[ \langle v_1 \circ \bR \circ v_1 \phi_j, \phi_k \rangle = \frac{1}{2} \delta^k_j + \frac{1}{4} ( \delta^2_{j-k} + \delta^{-2}_{j-k}) - \frac{1}{4} (\delta^1_j + \delta^{-1}_j)(\delta^1_k + \delta^{-1}_k). \]
Hence, with respect to the basis $(\phi_j)_{j\in\N}$, the operator $\bR + v_1 \circ \bR \circ v_1$ can be represented by the symmetric, ``double-infinite matrix'' $\bY$:
\NiceMatrixOptions{cell-space-limits = 1pt}

{\scriptsize
\begin{equation}\label{op:T}
\setlength\arraycolsep{3pt}
\bY =
\begin{bNiceArray}{cccccc|c|cccccc}[margin,last-row,last-col]
\ddots & \ddots & \ddots & \ddots & \ddots & & & & & & & & &\\
 & \tfrac14 & 0 & 1+\tfrac12 & 0 & \tfrac14 & & & & & & & &\\ 
 & & \tfrac14 & 0 & 1+\tfrac12 & 0 & \tfrac14 & & & & & & &\\ 
 &  & & \tfrac14 & 0 & 1+\tfrac14 & 0 & 0 &  &  &  &  & & k=-1 \\
\hline
& & & & \tfrac14 & 0 & \tfrac12 & 0 & \tfrac14 & & & & & k=0 \\ 
\hline
 & & & & & 0 & 0 & 1+\tfrac14 & 0 & \tfrac14 & & & & k=1 \\ 
 & & & & & & \tfrac14 & 0 & 1+\tfrac12 & 0 & \tfrac14 & & & \\ 
 & & & & & & & \tfrac14 & 0 & 1+\tfrac12 & 0 & \tfrac14 & & \\ 
 & & & & & & & & \ddots & \ddots & \ddots & \ddots & \ddots &\\
 & & j=-4 & j=-3 & j=-2 & j=-1 & j=0 & j=1 & j=2 & j=3 & j=4 & & & 
\end{bNiceArray}
\end{equation}
}
We finally establish that $\bY \ge \kappa \bI$ for $\kappa := \frac{3-\sqrt{5}}{2} \approx 0.38$ by subtracting
suitable non-negative $3 \times 3$-blocks from $\bY$ such that only a diagonal matrix with entries $\ge \kappa$ remains. More precisely, subtracting
\[
\begin{bmatrix}
\frac{1}{4} - \frac{\kappa}{2} & 0 & \frac{1}{4} \\
0 & 0 & 0 \\
\frac{1}{4} & 0 & \frac{5}{4} - \frac{\kappa}{2}
\end{bmatrix}
\ge 0
\]
with its upper left entry aligned at the indices $(0,0)$, $(1,1)$, $(2,2)$, $\dotsc$ as well as its antidiagonal transpose with lower right entry aligned at the indices $(0,0)$, $(-1, -1)$, $(-2, -2)$, $\dotsc$ yields
\[ \bY \ge \diag ( \dotsc, \kappa, \kappa, 1 + \frac{\kappa}{2}, \kappa, 1 + \frac{\kappa}{2}, \kappa, \kappa, \dotsc ) \ge \kappa \bI\,, \]
and hence \eqref{est1}. 
Note that $\kappa$ was chosen as the largest value such that the above $3 \times 3$-matrix is nonnegative.
\makered{More precisely, $\kappa=\frac{3-\sqrt5}{2}$ is the optimal constant for the estimate $\bY \ge \kappa \bI$. This can be seen by computing numerically the minimal eigenvalue for a sequence of finite dimensional matrices of form \eqref{op:T}, with increasing dimension.}
\\

\noindent
\makered{\underline{Part (b):}
We proceed as in \S4.3 of \cite{AAC16} for the 1D BGK equation and introduce for each mode a weighted norm that exhibits purely exponential decay. As in part (a) we can identify $\bJ_\bn$ (up to a rotation) with $|\bn| \bJ_{(1,0)}$. For the latter operator we have  
$$
  \bJ_{(1,0)}\phi=-iv_1\phi=-\frac{i}{2}(e^{i\varphi}+e^{-i\varphi})\phi
$$
with the basis representation
$$
  \langle \bJ_{(1,0)}\phi_j,\phi_k\rangle=-\frac{i}{2}(\delta_{j+1}^k+\delta_{j-1}^k).
$$
Hence, it can be represented by the skew-symmetric, ``doubly-infinite'' matrix
\begin{equation}
\begin{split}\label{J10-repr}
&-\frac{i}{2} \left[ \begin{array}{c|c|c|c|c}
\ddots & & & & \\
 \hline
 1 & 0 & 1 & & \\
 \hline
 & 1 & 0 & 1 & \\
 \hline
 & & 1 & 0 & 1 \\
 \hline
 & & & & \ddots
\end{array}\right] 
{\scriptsize 
\begin{array}{l}
\\[-0.4em]  \\[0.6em] k=0 \\[0.6em] \\[0.5em]
\end{array} } , \\
& {\scriptstyle\hspace{6.3em} j=0 } 
\end{split}
\end{equation}
while the operator $\bR$ corresponds to $\diag(...,1,1,0,1,1,...)$. Up to the mentioned rotation of the coordinate system, the evolution of $f_\bn$ is determined by
$$
  \bC_{|\bn|}:= \bR-|\bn|\bJ_{(1,0)} \in \mathcal B(L^2(\sphere^1)),\quad \bn\in\Z^2\setminus\{0\}\,.
$$
For each $\bn$ we now introduce a weighted norm that is equivalent to $\|.\|_{L^2(\sphere^1)}$. The weight matrix $\bY_ {|\bn|}$ will be chosen such that it can compensate the lack of coercivity of $\bR$ in the component $j=k=0$ by including the ``rotational effect'' of $\bJ_\bn$.

With the ``doubly-infinite'' weight matrices
\begin{equation}
\begin{split}\label{Pn}
&\bY_{|\bn|}:= \left[ \begin{array}{cc|c|c|cc}
\ddots & & & & & \\
 & 1 & & & & \\
 \hline
 & & 1 & -i\frac{\alpha}{|\bn|} & & \\
 \hline
 & & \phantom{-}i\frac{\alpha}{|\bn|} & 1 & &  \\
 \hline
 & & & & 1 &  \\
 & & & & & \ddots
\end{array}\right] 
{\scriptsize 
\begin{array}{l}
\\[-0.7em] k=-1  \\[0.65em] k=0 \\[0.65em] k=1 \\[0.65em] k=2
\end{array} } , \quad |\bn|\ge1 \,,\\
& {\scriptstyle\hspace{5em} j=-1 \scriptstyle\hspace{1.8em} j=0 \scriptstyle\hspace{1.8em} j=1 \scriptstyle\hspace{1.8em} j=2 } 
\end{split}
\end{equation}
with some parameter $\alpha\in\R$ to be determined, we want to satisfy the Lyapunov inequalities
\begin{equation}\label{Lyap-ineq}
    \bC^*_{|\bn|}\bY_{|\bn|}+\bY_{|\bn|}\bC_{|\bn|}  \ge 2 \lambda_0 \bY_{|\bn|}\,,\quad |\bn|\ge1\,.
\end{equation}
The left-hand side reads
\begin{equation}
\begin{split}\label{Lyap-inequ-lhs}
& \left[ \begin{array}{cc|c|c|c|c|cc}
\ddots & & & & & & &\\
 & 2 & & & & & & \\
 \hline
 & & 2 & & -\frac{\alpha}{2} & & & \\
 \hline
 & & & \alpha & -i\frac{\alpha}{|\bn|} & \frac{\alpha}{2} & &  \\
 \hline
 & &-\frac{\alpha}{2} & i\frac{\alpha}{|\bn|} & 2-\alpha & & & \\
 \hline
 & & & \frac{\alpha}{2} & & 2 &  &  \\
 \hline
 & & & & & & 2 &  \\
 & & & & & & & \ddots
\end{array}\right] 
{\scriptsize 
\begin{array}{l}
\\[-0.7em] k=-1  \\[0.65em] k=0 \\[0.65em] k=1 \\[0.65em] k=2
\end{array} } ,\\
& {\scriptstyle\hspace{4.7em} j=-1 \scriptstyle\hspace{0.8em} j=0 \scriptstyle\hspace{1.4em} j=1 \scriptstyle\hspace{1em} j=2 } 
\end{split}
\end{equation}
and we choose $\alpha=\frac12$ (which is close to the numerically determined optimal value $\alpha_0\approx0.45$). Then we estimate the essential $4\times4$ submatrix of \eqref{Lyap-inequ-lhs}, i.e.\ its not purely diagonal part:
\begin{equation*}
\begin{split}
 \bZ_{|\bn|}:=\left[ \begin{array}{cccc}
 2 & 0 & -\frac{1}{4} & 0 \\
 0 & \frac12 & -i\frac{1}{2|\bn|} & \frac{1}{4} \\
 -\frac{1}{4} & i\frac{1}{2|\bn|} & \frac32 & 0 \\
  0 & \frac{1}{4} & 0 & 2   
 \end{array}\right] 
&\ge 
 \left[ \begin{array}{cccc}
 2 & 0 & -\frac{1}{4} & 0 \\
 0 & 1-\frac{1}{\sqrt2} & 0 & \frac{1}{4} \\
 -\frac{1}{4} & 0 & 1-\frac{1}{\sqrt2} & 0 \\
  0 & \frac{1}{4} & 0 & 2   
 \end{array}\right] \\
& \ge 3\lambda_0\bI \ge2\lambda_0 \left[\bY_{|\bn|}\right]_{-1\le j,k\le2} \,,\quad |\bn|\ge1\,, 
\end{split}
\end{equation*}
and this verifies \eqref{Lyap-ineq}. Here we have used that the eigenvalues of $\bY_{|\bn|}$ are $1\pm\frac{1}{2|\bn|}$ and~$1$; in particular, those of $\bY_1$ are $\frac12$, $\frac32$ and $1$. We note that $3\lambda_0$ is also the smallest eigenvalue of $\bZ_1$, and the smallest eigenvalues of $\bZ_{|\bn|}$ with $|\bn|>1$ are larger.

Using the Fourier decomposition 
$$
  f_\bn(\varphi,t) = \sum_{j \in \Z} \hat f_{\bn j}(t) \phi_j(\varphi) 
$$
we introduce the weighted norm
$$
  \left\|\big(\hat f_{\bn j}\big)_{j\in\Z}\right\|^2_{\bY_{|\bn|}} :=
  \sum_{j,n\in\Z} \hat f_{\bn j} \big(\bY_{|\bn|}\big)_{jk} \overline{\hat f_{\bn k}}\,.
$$
Using the time evolution of $f_\bn$ and \eqref{Lyap-ineq}, we obtain for $|\bn|\ge1$ (proceeding as in (4.6), (4.7) of \cite{AAC16}):
\begin{eqnarray*}
    \|f_\bn(t)\|_{L^2(\sphere^1)}^2
    &\le& \frac{1}{1-\frac{1}{2|\bn|}}  
    \left\|\big(\hat f_{\bn j}(t)\big)_{j\in\Z}\right\|^2_{\bY_{|\bn|}}
    \le \frac{1}{1-\frac{1}{2|\bn|}}  
    \left\|\big(\hat f_{\bn j}(0)\big)_{j\in\Z}\right\|^2_{\bY_{|\bn|}} e^{-2\lambda_0 t}
\\
    & \le& \frac{1+\frac{1}{2|\bn|}}{1-\frac{1}{2|\bn|}}  
    \|f_\bn(0)\|^2_{L^2(\sphere^1)} e^{-2\lambda_0 t}\,.
\end{eqnarray*}
Noting that $\frac{1+\frac{1}{2|\bn|}}{1-\frac{1}{2|\bn|}} =\frac{2|\bn|+1}{2|\bn|-1}\le 3$ for $|\bn|\ge1$, and taking the minimum of this estimate and the initial norm $\norm{f_\bn(0)}^2_{L^2(\sphere^1)}$ kept constant (due to the semi-dissipativity of \eqref{lorentz_mode_n}) completes the proof of Part (b).\\
}

\noindent
\makered{\underline{Part (c):}
To derive the uniform estimate \eqref{Pn-decay} we combine a short-term decay estimate for the initial phase $[0,\frac{\tau}{|\bn|}]$ (which shrinks w.r.t.\ $|\bn|$) with the long-term decay estimate \eqref{unif-decay} for the remaining time interval $[\frac{\tau}{|\bn|},\tau]$, see Figure~\ref{fig:sketch_of_proof}. The proof of the uniform estimate~\eqref{Pn-decay} is given in Appendix~\ref{sec:Proof_of_part_c}.

\begin{figure}
  \centering
  \includegraphics[width=0.8\textwidth]
  {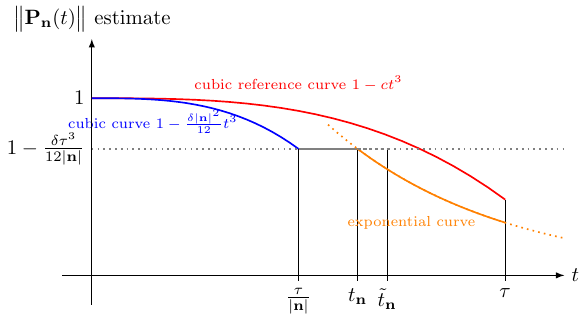}
  \caption{To derive the uniform estimate $\|\bP_\bn(t)\|_{\cB(L^2(\sphere^1))} \le 1-c t^3$ for $0\le t\le \tau$ in~\eqref{Pn-decay}, we combine a short-term decay estimate for the initial phase $[0,\frac{\tau}{|\bn|}]$  (that shrinks w.r.t.\ $|\bn|$) with the long-term decay estimate \eqref{unif-decay} for the remaining time interval $[\frac{\tau}{|\bn|},\tau]$.}
  \label{fig:sketch_of_proof}
\end{figure}
}

\end{proof}

\subsection{\makered{Analysis as infinite dimensional dynamical system}}\label{full-system}

\makered{We remark that both the decay estimates \eqref{unif-decay}, \eqref{Pn-decay} and the hypocoercivity estimate~\eqref{est1} are uniform with respect to $\bn \in \Z^2\setminus\{0\}$. This will allow us to combine all the modal results from \S\ref{mode-ana} to establish the analogous results for the full system \eqref{lorentz}.}

\makered{To this end we} define the hypocoercivity index directly with the operators $\bR(f) \coleq f - \tilde f$, $\bJ = -\bv \cdot \nabla_\bx$ acting on phase space functions $f = f(\bx,\bv)$, i.e., without an $\bx$-modal decomposition. Denoting $\cH:=L^2(\torus^2\times\sphere^1)$, the operator $\bJ$ is unbounded with
\[ \cD_\circ \coleq \Big\{ f \in \cH \ \Big|\ \frac{\partial^2 f}{\partial \bx^2} \in \cH  \Big\} \subseteq \cD(\bJ) \]
since $\bR \bJ(\cD_\circ) \subseteq \cD (\bJ)$. 
Thus, Definition \ref{def:op:HC-index} for $m_{HC}=1$ can be extended to the example of the Lorentz kinetic equation with $\bJ$ unbounded by requiring that
\begin{equation}\label{est2}
\bR + \bJ \bR \bJ^* \ge \kappa \bI
\end{equation}
holds on $\cD_\circ \cap \{ \int_{\torus^2 \times \sphere^1} f \,\ud \bx \ud \bv = 0 \}$. The latter restriction is necessary, since the Lorentz kinetic equation is mass preserving or, equivalently, for the zeroth mode, $\bR - \bJ_0 = \bR$ is coercive only on $(\ker \bR)^\perp \subseteq L^2(\sphere^1)$.

This allows to refine Theorem~\ref{lorentz-HC} \makered{by combining the modal results from Lemma~\ref{lem1}: }

\begin{proposition}
Let $d=2$. Then:
\begin{enumerate}
    \item[(a)] The operator $\bC = \bR - \bJ$ is hypocoercive on $\cD_\circ \cap \{ \int_{\torus^2 \times \sphere^1} f \,\ud \bx \ud \bv = 0 \}$ with index $m_{HC}=1$. In particular, inequality~\eqref{est2} holds with $\kappa = \frac{3-\sqrt{5}}{2}$.
    \item[(b)] For any $f_0 \in \cH$ the solution of \eqref{lorentz} satisfies
\begin{equation}\label{f-decay}
  \left\| f(t) - f_\infty \right\|_{\cH} 
  \le \sqrt3 e^{-\lambda_0 t} \left\| f_0 - f_\infty\right\|_{\cH}, 
\end{equation} 
where $f_\infty:=\frac{1}{(2\pi)^3} \int_{\torus^2 \times \sphere^{1}} f_0(\bx,\bv) \,\ud \bv \,\ud \bx$ denotes the unique equilibrium corresponding to $f_0$ and $\lambda_0$ is given in Lemma~\ref{lem1}(b).
\item[(c)] The propagator for \eqref{lorentz}, $\bP(t):=e^{-\bC t}$, satisfies
\begin{equation}\label{P-decay}
    1-\tilde c t^3 \le \|\bP(t)-\Pi_\infty\|_{\cB(\cH)}\le 1-c t^3,\quad 0\le t\le \tau,
\end{equation}
where the constants $c,\,\tau>0$ are from Lemma \ref{lem1}(c), $\tilde c>0$ is given in the proof below and $\Pi_\infty$ denotes the $\cH$-orthogonal projection on $\Span[f_\infty]$.
\end{enumerate}
\end{proposition}

\begin{proof}
\makered{\underline{Part (a):}}
For the $\bx$-modal decomposition of $f \in \cH$ we use
\begin{equation}\label{modal-deco}
  f(\bx,\varphi) = \frac{1}{2\pi} \sum_{\bn \in \Z^2} f_\bn(\varphi) e^{i\bn \cdot \bx} 
\end{equation}
with $\bv = (\cos \varphi, \sin \varphi) \in \sphere^1$. Denoting the inner product on \makered{$\cH=L^2(\torus^2 \times \sphere^1)$} by $\llangle ., . \rrangle$ we compute for (real-valued) $f \in \cD_\circ \cap \{ f \ \big| \ \int_{\torus^2 \times \sphere^1} f \, \ud \bx\, \ud \bv = 0 \}$:
\begin{align}
\nonumber \llangle \bR f, f \rrangle &= \int_{\torus^2} \int_0^{2\pi} \Big[ f(\bx, \varphi) - \frac{1}{2\pi} \int_0^{2\pi} f(\bx,\psi)\,\ud \psi\Big]  \overline{f(\bx,\varphi)} \,\ud \varphi\, \ud \bx \\
\nonumber &= \frac{1}{(2\pi)^2} \sum_{\bn,\Bm \in \Z^2} \int_0^{2\pi} \int_{\torus^2} e^{i(\bn-\Bm)\cdot \bx} \,\ud \bx\, \big[ f_\bn(\varphi) - \tilde f_\bn \big] \overline{f_\Bm(\varphi)}\,\ud \varphi \\
\label{est3} &= \sum_{\bn \in \Z^2} \int_0^{2\pi} \big[ f_\bn(\varphi) - \tilde f_\bn\big]\, \overline{f_\bn}(\varphi)\,\ud \varphi = \sum_{\bn \in \Z^2} \langle \bR f_\bn, f_\bn \rangle\,.
 \end{align}
The computation from Lemma \ref{lem1} also carries over to the operator $\bJ$, since the spatial modes $e^{i \bn\cdot \bx}$ are eigenfunctions of  $\bv \cdot \nabla_\bx $. Using
\[ \bR (\bJ^* f) = \frac{i}{\sqrt{2\pi}} \sum_{\bn \in \Z^2} e^{i \bn \cdot \bx} \bR (\bv \cdot \bn f_\bn) \]
we compute
\begin{align*}
\llangle \bJ \bR(\bJ^* f), f \rrangle &= - \int_{\torus^2} \int_0^{2\pi} \bR (\bv \cdot \nabla_\bx f) \bv \cdot \nabla_\bx f \,\ud \varphi \,\ud \bx \\
&= - \frac{1}{(2\pi)^2} \sum_{\bn,\Bm \in \Z^2} \int_0^{2\pi} \int_{\torus^2} e^{i(\bn-\Bm)\cdot \bx} \,\ud \bx \,\bR(\bv \cdot \bn f_\bn) \bv \cdot \Bm \overline{f_\Bm} \,\ud \varphi \\
&= - \sum_{\bn \ne 0} \int_0^{2\pi} \bR (\bv \cdot \bn f_\bn) \bv \cdot \bn \overline{f_\bn} \,\ud \varphi = \sum_{\bn \ne 0} \langle \bJ \bR \bJ^* f_\bn, f_\bn \rangle \,.
\end{align*}
Adding this line to \eqref{est3} yields the claim \eqref{est2}:
\begin{align*}
\llangle (\bR + \bJ \bR \bJ^*) f, f \rrangle &= \langle \bR f_0, f_0 \rangle + \sum_{\bn \ne 0} \langle (\bR + \bJ \bR \bJ^*)f_\bn, f_\bn \rangle \\
& \ge \norm{f_0}^2  + \kappa \sum_{\bn \ne 0} \norm{f_\bn}^2 \ge \kappa \sum_{\bn \in \Z^2} \norm{f_\bn}^2=\kappa \|f\|^2_\cH 
\end{align*}
for all $f \in \cD_\circ \cap \{ f \ \big| \ \int_{\torus^2 \times \sphere^1} f \, \ud \bx\, \ud \bv = 0 \}$.

\medskip\noindent
\makered{\underline{Part (b):}
From the $\bx$-modal decomposition \eqref{modal-deco} of $f \in \cH$ we have
$$
  \|f\|_\cH^2 = \sum_{\bn\in\Z^2} \|f_\bn\|^2_{L^2(\sphere^1)}\,,
$$
and \eqref{f-decay} follows directly from \eqref{unif-decay}.

\medskip\noindent
\underline{Part (c):}
As in Part (b), the upper bound of \eqref{P-decay} follows directly from \eqref{Pn-decay}. The lower bound follows from $\|f\|_\cH\ge \|f_{(1,0)}\|_{L^2(\sphere^1)}$ along with the cubic lower bound \eqref{starstar}, e.g.\ for the $(1,0)$-mode.
}
\end{proof}

We remark that the above analysis can be extended to the case $d=3$ and higher using the spherical harmonics basis on $\sphere^{d-1}$.


\section{Further examples}\label{sec_furtherexamples}

In this section we present some examples to further illustrate the concept of hypocoercivity and to delineate the differences to the finite-dimensional case.

For accretive matrices~$\bC$, the spectral norm of the associated propagator (matrix exponential) $e^{-\bC t}$ is real analytic for~$t\in[0,t_0)$ for some $t_0>0$.
The following example shows that the time interval~$t\in[0,t_0)$ can be arbitrarily short:
\begin{example}[Propagators whose norm is real analytic on arbitrarily short intervals]
\label{example:family_Ck}
For $k\in\N$, consider matrices~$\bC_k \in\R^{2\times 2}$ of the form
\begin{equation}\label{C_k}
 \bC_k
=\begin{bmatrix}
 0 & k \\
-k & 1
 \end{bmatrix} \in \R^{2\times 2}
\ , \quad
(\bC_k)_H =\diag(0,1)
\ , \quad
(\bC_k)_S =k \begin{bmatrix} 0 & 1 \\ -1 & 0 \end{bmatrix} .
\end{equation}
The matrices $\bC_k$, $k\in\N$ are accretive with eigenvalues $\lambda_\pm=\tfrac12 (1 \pm i \sqrt{4k^2-1})$.
Hence, all solutions of system \eqref{lineveq} with $\bC_k$ decay with rate $1/2$. 
Due to~\cite[Theorem 3.7]{AAS19}, for $k\in\N$, we derive 
\begin{equation} \label{long-t}
 \big\|e^{-\bC_k t}\big\|_2 
\leq \sqrt{\frac{2k+1}{2k-1}}\ e^{-t/2}
\qquad \text{for } t\geq 0\,,
\end{equation}
where $\|\cdot\|_2$ denotes the spectral norm on $\R^{2\times2}$.
Thus, the matrices~$\bC_k$, $k\in\N$ are hypocoercive and have HC-index $\mHC=1$.
In case of $2\times 2$ matrices, explicit expressions for the propagator norm are available, see~\cite[Proposition 4.2]{AAS19}:
For $t\geq 0$, the propagator norm $\|e^{-\bC_k t}\|^2_2$ satisfies
\begin{equation} \label{est:sandwich}
 e^{-t}\ m_-(t)
\leq \big\|e^{-\bC_k t}\big\|_2^2 
= e^{-t}\ m_+(t) 
\end{equation}
with
\[
 m_{\pm}(t) 
:= \pm \Bigg( \sqrt{\frac{(1-\alpha^2 \cos(\delta t))^2}{(1-\alpha^2)^2}-1} \pm \frac{1-\alpha^2 \cos(\delta t)}{1-\alpha^2} \Bigg)\,,
\]
where $\delta:=\sqrt{4k^2-1}$ and $\alpha:=1/(2k)$, see~\cite[Proposition 4.2]{AAS19}.
For fixed $t\geq 0$, the functions $h_{\pm}(t) :=e^{-t} m_{\pm}(t)$ are the eigenvalues of $e^{-\bC_k^* t}e^{-\bC_k t}$, which intersect (by touching each other) for the first (positive) time at $t_k=2\pi/\sqrt{4k^2-1}$. There, the functions $h_\pm$ have a kink.
For $k\to\infty$, the intersection times $t_k$ converge (monotonically) to $0$.

\cite[Theorem 2.7]{AAC2} yields the following Taylor expansion of the propagator norm, for $k\in\N$:
\begin{equation} \label{short-t}
 \big\|e^{-\bC_k t}\big\|_2
=1 -\frac{k^2}{12}\ t^3 +\mathcal{O}(t^4)
 \qquad\text{for } t\in[0,t_k).
\end{equation}
For $k\to\infty$, the matrices $\bC_k$ become unbounded and the multiplicative factor $k^2/12$ in~\eqref{short-t} diverges, and the Taylor expansion~\eqref{short-t} holds only on shrinking intervals, since $\lim_{k\to\infty} t_k =0$.
\end{example}

\bigskip
\begin{example}[A bounded operator which is not hypocoercive]
\label{example:AAC}
Following~\cite[Example 2.6]{AAC2}, we consider an evolution equation~\eqref{lineveq} with block-diagonal operator~$\bC=\diag(\bC_1,\bC_2,\ldots)$, where each block matrix~$\bC_k$, $k\in\N$ is of the form
\begin{equation}\label{E_k}
 \bC_k
=\begin{bmatrix}
 0  & 1 & & & \\
-1 & \ddots & \ddots & & \\
   & \ddots & \ddots & 1 & \\
   & & -1 & 0 & 1 \\
   & & & -1 & 1
 \end{bmatrix} \in \R^{k\times k}\,,
\end{equation}
with $\bC_k=\bR_k -\bJ_k$ such that $\bR_k =\diag(0,0,\ldots,0,1)$ and 
$\bJ_k =-\tridiag(-1,0,1)$.
The accretive matrices $\bC_k$, $k\in\N$ are hypocoercive with HC-index~$\mHC(\bC_k) =k-1$, see~\cite[Example 2.6]{AAC2}.
Thus, the bounded accretive operator~$\bC=\diag(\bC_1,\bC_2,\ldots)$ cannot have a finite hypocoercivity index (as in Definition \ref{def:op:HC-index}). 
Hence, due to Theorem \ref{cor:exp-stable}, $\bC$ is not hypocoercive.

Another way to realize that $\bC$ is not hypocoercive is to show that the (sequence of) spectral gaps of $\bC_k$, $\mu_k$ converges to $0$ as $k\to\infty$:
The hypocoercive matrices~$\bC_k$, $k\in\N$ satisfy the estimate 
\begin{equation} \label{est:Ek}
 \| e^{-\bC_k t} \|_2 
\leq c_k e^{-\mu_k t} , 
\qquad
 t\geq 0 ,
\end{equation}
for constants $c_k\geq 1$ and $\mu_k>0$, $k\in\N$. 
The matrices $\bC_k$, $k\in\N$ have trace  $\trace \bC_k=1$, and all eigenvalues $\lambda^{\bC_k}\in\sigma(\bC_k)$ have non-negative real part, $\Re\lambda^{\bC_k}\geq 0$. 
Using the arithmetic-geometric mean inequality yields
\begin{equation}
 0
\leq \mu_k= \min_j \big(\Re \lambda_j^{\bC_k}\big)
\leq \Big( \prod_{j=1}^k \Re\big(\lambda_j^{\bC_k}\big) \Big)^{1/k}
\leq \frac1k \sum_{j=1}^k \Re\big(\lambda_j^{\bC_k}\big)
= \frac{\trace \bC_k}{k}
= \frac1k .
\end{equation}
Therefore, the (sequence of) spectral gaps $\mu_k$ converges to $0$ as $k\to\infty$. 
Since these $\mu_k$ are the largest possible decay rates in \eqref{est:Ek}, the accretive operator~$\bC=\diag(\bC_1,\bC_2,\ldots)$ is not hypocoercive.
\end{example}
\bigskip

\begin{example}[An unbounded operator with "infinite" hypocoercivity index]
\label{example:AAC_rescaled}
In~\cite[Example 2.6]{AAC2}, the non-hypocoercive operator $\bC=\diag(\bC_k; k\in\N)$ with $\bC_k$ given in~\eqref{E_k} is modified to obtain a hypocoercive operator (with infinite hypocoercivity index):
Consider, instead of $\bC=\diag(\bC_k; k\in\N)$, the block-diagonal operator $\tbC=\diag(\tbC_k; k\in\N)$ with rescaled matrices $\tbC_k \coleq r_k \bC_k$ for $r_k>0$ to be chosen below. 
$\tbC$ is considered on $\cH:=\bigoplus_{k=1}^\infty \R^k$ with $\|\bx\|^2:=\sum_{k=1}^\infty \|x_k\|^2$ for $\bx=(x_1,x_2,...)\in\cH$ and dense domain $\cD(\tbC):=\{\bx\in\cH\ :\ (r_1x_1,r_2x_2,r_3x_3,...)\in\cH\}$.
Let $\|\cdot\|_2$ denote the spectral norm on $\R^{k\times k}$. 
Then, the propagator $\mP_k(t):=e^{-\tbC_k t}$ satisfies
\[
 \|\bP_k(1)\|_2
=\|e^{-\tbC_k}\|_2 
=\|e^{-r_k\bC_k}\|_2 
\leq c_k e^{-r_k\mu_k} 
\leq \frac1e \quad \text{for }\quad r_k 
= \frac{1+ \log c_k}{\mu_k}\ 
\]
due to \eqref{est:Ek}.
Therefore, making this choice of $r_k$ and $t_0=1$, $\|e^{-\tbC t_0}\| \leq 1/e$, so that Proposition~\ref{prop:UniExpStableSG}\ref{prop:UniExpStableSG:c} is satisfied for $t_0 =1$. 
Thus, the accretive operator~$\tbC$ is hypocoercive. 

However, since $c_k\geq 1$ and $\mu_k\to0$,  
we deduce 
\[
 \lim_{k\to\infty} r_k 
=\lim_{k\to\infty} \frac{1+ \log c_k}{\mu_k}
\geq \lim_{k\to\infty} \frac1{\mu_k}
=\infty.
\]
Hence, the operator $\tbC$ 
is unbounded.
\end{example}
\bigskip

{\bfseries Declarations of interest. } 
none


{\bfseries Acknowledgments. } 
The first two authors (FA, AA) were supported by the Austrian Science Fund (FWF) via
the FWF-funded SFB \# F65. The third author (VM)  was supported by Deutsche Forschungsgemeinschaft (DFG, German Research Foundation) CRC 910 \emph{Control of self-organizing nonlinear systems: Theoretical methods and concepts of application}: Project No.~163436311.

\printbibliography
\appendix

%
%
\section{Appendix}\label{sec:appendix}

In this appendix we collect some conditions that are equivalent to hypocoercivity in the finite dimensional case; 
here they are presented in the Hilbert space case. 
But, presently, they are not used for the rest of the paper.

\begin{lemma}\label{lemmai}
Let $\bC = \bR - \bJ \in \cB(\cH)$ with $\bR^* = \bR \ge 0$ and $\bJ^* = - \bJ$. Then, for any $m \in \N$ and $x \in \cH$ the following four assertions are equivalent:
\begin{enumerate}[label=(\roman*)]
  \item\label{lemmai_i} $\langle \bC^j \bR (\bC^*)^j x, x \rangle = 0$ for all $0 \le j < m$.
  \item\label{lemmai_ii} $\langle \bJ^j \bR (\bJ^*)^j x, x \rangle = 0$ for all $0 \le j < m$.
  \item\label{lemmai_iii} $\langle (\bC^*)^j \bR \bC^j x, x \rangle = 0$ for all $0 \le j < m$.
  \item\label{lemmai_iv} $\langle \bC_j^* \bC_j x, x \rangle = 0$ for all $0 \le j < m$ with $\bC_0 \coleq \sqrt{\bR}$, and the commutators $\bC_{j+1} \coleq [\bJ, \bC_j]$, $j \in \N_0$.
\end{enumerate}
If these hold, then we also have
\begin{enumerate}[label=(\roman*),resume]
\item\label{star}
$\langle \bC^m \bR (\bC^*)^m x, x \rangle = \langle \bJ^m \bR (\bJ^*)^m x, x \rangle = \langle (\bC^*)^m \bR \bC^m x, x \rangle = \langle \bC_m^* \bC_m x, x \rangle$.
\end{enumerate}
\end{lemma}

\begin{proof}
The equivalence of \ref{lemmai_i}--\ref{lemmai_iv} is trivial for $m=1$. We proceed by induction: 
Assuming that it holds for some $m$ we shall show \ref{star}.

By assumption, from \ref{lemmai_ii} we have $x \in \ker \sqrt{\bR} (\bJ^*)^j = \ker \bR (\bJ^*)^j$ for all $j<m$. We consider
\begin{align*}
\langle \bC^m \bR (\bC^*)^m x, x \rangle &= \langle (\bR - \bJ)^m \bR (\bR - \bJ^*)^m x, x \rangle \\
&= \langle \bJ^m \bR (\bJ^*)^m x, x \rangle + \dotsc,
\end{align*}
where each remaining term is of the form
\begin{equation}\label{termu}
\pm \langle \bT_1 \bT_2 \dotsm \bT_m \bR \bS_1^* \bS_2^* \dotsm \bS_m^* x, x \rangle
\end{equation}
with $\bT_i, \bS_i \in \{ \bJ, \bR \}$ for $i = 1, \dotsc, m$, where at least one $\bT_i$ or $\bS_i$ is equal to $\bR$. If $\bS_i = \bR$ for at least one $i$ we take the largest $i$ with this property, which means that $\bS_{i+1}, \dotsc, \bS_m = \bJ$, hence $\bS_i^* \bS_{i+1}^* \dotsm \bS_m^* x = \bR (\bJ^*)^{m-i}x = 0$ and \eqref{termu} vanishes. If $\bT_i = \bR$ for at least one $i$ we take the smallest $i$ with this property, which means that $\bT_1, \dotsc, \bT_{i-1} = \bJ$, hence \eqref{termu} equals
\[ \pm \langle \bT_{i+1} \dotsm \bT_m \bR \bS_1 \bS_2 \dotsm \bS_m x, \bR (\bJ^*)^{i-1} x \rangle = 0. \]
As all terms of the form \eqref{termu} vanish, we obtain
\[ \langle \bC^m \bR (\bC^*)^m x, x \rangle = \langle \bJ^m \bR (\bJ^*)^m x, x \rangle. \]

The equality $\langle \bJ^m \bR (\bJ^*)^m x, x \rangle = \langle (\bC^*)^m \bR \bC^m x, x \rangle$ is seen in exactly the same way.

For the last equality in \ref{star} we note that, by definition, $\langle \bC_m x, \bC_m x \rangle$ is given as a sum of terms of the form
\[  \pm \langle \bT_1 \bT_2 \dotsm \bT_{m+1} \bS_1 \bS_2 \dotsm \bS_{m+1} x, x \rangle \]
with $\bT_i, \bS_i \in \{ \bJ, \sqrt{\bR} \}$ for $i = 1, \dotsc, m+1$ and exactly one of the $\bT_i$ and one of the $\bS_i$ are equal to $\sqrt{\bR}$. By assumption, all of these terms vanish except
\[ \langle \sqrt{\bR} \bJ^m x, \sqrt{\bR} \bJ^m x \rangle = \langle (\bJ^*)^m \bR \bJ^m x, x \rangle = \langle \bJ^m \bR (\bJ^*)^m x, x \rangle. 
\]
Hence \ref{star} holds, which also gives equivalence of the assertions \ref{lemmai_i}--\ref{lemmai_iv} for $m+1$. 
\end{proof}

From this, the following reformulations and modifications of \ref{p2} are evident.

\begin{corollary} \label{cor:Operators:HC_equivalences}
Let $\bC = \bR - \bJ \in \cB(\cH)$ with $\bR^* = \bR \ge 0$ and $\bJ^* = - \bJ$ and $m \in \N_0$. Then, the following assertions are equivalent:
\begin{enumerate}[label=(\roman*)]
\item $\displaystyle \sum_{j=0}^m \bC^j \bR (\bC^*)^j > 0$,
\item $\displaystyle \sum_{j=0}^m \bJ^j \bR (\bJ^*)^j > 0$,
\item $\displaystyle \sum_{j=0}^m (\bC^*)^j \bR \bC^j > 0$,
\item $\displaystyle \sum_{j=0}^m \bC_j^* \bC_j > 0$.
\end{enumerate}
\end{corollary}

Next, we consider similar variants of the equivalence \ref{p1} $\Longleftrightarrow$ \ref{p2}.

\begin{lemma}\label{lemmaii}Let $\bC = \bR - \bJ \in \cB(\cH)$ with $\bR^* = \bR \ge 0$ and $\bJ^* = -\bJ$ and $m \in \N_0$. Then
\begin{align*}
\sum_{j=0}^m \bJ^j \bR (\bJ^*)^j > 0 & \Longleftrightarrow \bigcap_{j=0}^m \ker \sqrt{\bR} (\bJ^*)^j = \{0 \} \Longleftrightarrow \overline{\Span} \bigcup_{j=0}^m \im ( \bJ^j \sqrt{\bR}) = \cH, \\
\sum_{j=0}^m \bC^j \bR (\bC^*)^j > 0 & \Longleftrightarrow \bigcap_{j=0}^m \ker \sqrt{\bR} (\bC^*)^j = \{0 \} \Longleftrightarrow \overline{\Span} \bigcup_{j=0}^m \im ( \bC^j \sqrt{\bR} )= \cH, \\
\sum_{j=0}^m (\bC^*)^j \bR \bC^j > 0 & \Longleftrightarrow \bigcap_{j=0}^m \ker \sqrt{\bR} \bC^j = \{0 \} \Longleftrightarrow \overline{\Span} \bigcup_{j=0}^m \im ( (\bC^*)^j \sqrt{\bR} )= \cH, \\
\sum_{j=0}^m \bC_j^* \bC_j > 0 & \Longleftrightarrow \bigcap_{j=0}^m \ker \bC_j = \{0 \} \Longleftrightarrow \overline{\Span} \bigcup_{j=0}^m \im \bC_j = \cH. 
\end{align*}
Moreover, all four rows are equivalent to each other.
\end{lemma}

\begin{proof} This is immediate as in Proposition \ref{propi}, taking into account that $\bC_j^* = \bC_j$ for the last equivalence.
\end{proof}

\section{Connection to linear systems and control theory}
\label{ssec:control}

In this appendix we discuss connections of our results to classical concepts in control theory.

Consider linear systems~$\Sigma(\fA,\fB,\fC,\fD)$ of the form
\begin{subequations} \label{LinearSystem}
\begin{align}
 \ddt z(t)
&=\fA z(t) +\fB u(t) , \qquad t\geq 0 , \label{LS:stateEq}
\\
 y(t)
&=\fC z(t) +\fD u(t) , \label{LS:outputEq}
\end{align}
\end{subequations}
with $\fA,\fB,\fC,\fD\in\cB(\cH)$, see~\cite[\S6]{CurtainZwart}.

A linear control system $\Sigma(\fA,\fB,-,-)$ of the form~\eqref{LS:stateEq} is said to be~\emph{exponentially stabilizable} if there exists a linear operator $\fK\in\cB(\cH)$ such that $\bAK:=\fA +\fB\fK$ generates a (uniformly) exponentially stable semigroup (i.e., $\bAK$ is hypocoercive), see~\cite[\S16.6]{Za20}, \cite[Definition 8.1.1]{CurtainZwart}.

If the accretive operator~$\bC\in\cB(\cH)$ with $\bC=\bR-\bJ$ is hypocoercive, then the linear systems~$\Sigma(\bJ,\bR,-,-)$ and~$\Sigma(\bJ,\sqrt{\bR},-,-)$ are exponentially stabilizable (using $\fK=-\bI$ and $\fK=-\sqrt{\bR}$, respectively).


Comparing controllability of linear systems in finite- and infinite-dimensional Hilbert space settings, one has to distinguish several different notions in the latter case such as exact and approximate controllability, see~\cite[Definition 6.2.1]{CurtainZwart}.

If a linear control system $\Sigma(\fA,\fB,-,-)$ of the form~\eqref{LS:stateEq} is exactly controllable then it is exponentially stabilizable.
However there exist approximately controllable linear systems which are not exponentially stabilizable, see~\cite[Example 8.1.2]{CurtainZwart}, \cite[Theorem 16.4]{Za20}.

If an accretive operator~$\bC\in\cB(\cH)$ with $\bC=\bR-\bJ$ is hypocoercive, then the linear systems~$\Sigma(\bJ,\bR,-,-)$ and~$\Sigma(\bJ,\sqrt{\bR},-,-)$ are exponentially stabilizable (using $\fK=-\bI$ and $\fK=-\sqrt{\bR}$, respectively).
Hence, in this case, the linear systems~$\Sigma(\bJ,\bR,-,-)$ and~$\Sigma(\bJ,\sqrt{\bR},-,-)$ have to be exactly controllable.

For other interesting results in the study of linear systems of the form~$\Sigma(\bJ,\bR,-,-)$, see~\cite[Theorem 16.5]{Za20} (noting that $\fA=\bJ$ generates a group), and ~\cite[\S~III.1.6]{BeDaPDeMi07}.

Conditions similar to the ones given in Proposition~\ref{propii}~\ref{p1'}--\ref{p3'} characterize \emph{asymptotically controllable/observable} linear systems~$\Sigma(\fA,\fB,\fC,\fD)$ of the form~\eqref{LinearSystem}, see~\cite[\S6]{CurtainZwart}:
\begin{itemize}
    \item 
Considering~\ref{p1'}, condition~\eqref{p1a'} appears in \cite[Theorem 6.2.25.d]{CurtainZwart} with $\fA=\bJ$ and $\fB=\sqrt{\bR}$.
Condition~\eqref{p1b'} appears in \cite[Corollary 6.2.26.d]{CurtainZwart} with $\fA=\bJ^*$ and $\fC=\sqrt{\bR}$.
Moreover, the equivalence between~\eqref{p1a'} and~\eqref{p1b'} is used in~\cite[Lemma 6.2.14.a, Corollary 6.2.21.a]{CurtainZwart}.
    \item
Considering~\ref{p2'}, the condition $\sum_{j=0}^\infty \bJ^j \bR (\bJ^*)^j > 0$ is related to~\cite[Corollary 6.2.15.b.i]{CurtainZwart} with $\fA=\bJ^*$ and $\fC=\sqrt\bR$ (arguing as in the proof of~\cite[Theorem 6.2.27]{CurtainZwart}).
    \item
Finally, condition~\ref{p3'} is related to the characterization of $\Sigma(\bJ^*,-,\sqrt\bR,-)$ being approximately observable in infinite time given in~\cite[Definition 6.2.18]{CurtainZwart} and~\cite[Lemma 6.2.22]{CurtainZwart}.
\end{itemize}
The equivalences between the statements follow from the given references, and the duality of controllability and observability notions, see e.g.~\cite[Lemma 6.2.14]{CurtainZwart} and \cite[Corollary 6.2.21]{CurtainZwart}.


Conditions similar to the ones given in Lemma~\ref{lem:Op-Equivalence}~\ref{B:G*_surjective}--\ref{B:Tmop} characterize \emph{exactly controllable/observable} linear systems~$\Sigma(\fA,\fB,\fC,\fD)$ of the form
\eqref{LinearSystem}, see~\cite[\S6]{CurtainZwart}:
The equivalence in each ``row" statement follows from~\cite[Theorem 6.2.6.a]{CurtainZwart} using arguments as in the proof of~\cite[Theorem 6.2.27]{CurtainZwart}.

\bigskip
Under the assumption~\eqref{span:infinity}, the linear system~\eqref{LinearSystem} with $\Sigma(\bJ^*,\sqrt\bR,-,-)$ is exactly controllable (see \cite[Theorem 6.2.6.a.iv]{CurtainZwart}).
Then, due to~\cite[Theorem 6.2.27]{CurtainZwart}, the identity~\eqref{span:m} follows. 
For earlier results, see~\cite{Fuhrmann} (mentioned in the proof) and~\cite{CurtainPritchard}.

The equivalence of (i)--(iv) in both Lemma~\ref{lemmai} and Corollary~\ref{cor:Operators:HC_equivalences} is related to~\cite[Lemma 6.2.5]{CurtainZwart}: For example, choose $\bA=\bJ-\bR$, $\bB=\sqrt\bR$ and $\bF=\sqrt\bR$ such that $\bA+\bB\bF=\bJ$.

%
%

\makered{
\section{Deferred proofs}
\label{sec:Proof_of_part_c}

\begin{proof}[Proof of Lemma~\ref{lem1}(c)] 
To derive the uniform estimate \eqref{Pn-decay} we combine a short-term decay estimate for the initial phase $[0,\frac{\tau}{|\bn|}]$ (that shrinks w.r.t.\ $|\bn|$) with the long-term decay estimate \eqref{unif-decay} for the remaining time interval $[\frac{\tau}{|\bn|},\tau]$. Hence, we structure this proof part in several steps.

In the following, $x$ represents an element $f_\bn\in L^2(\sphere^1)=:H$, and $\sphere_H \coleq \{ x\in H : \norm{x}_H =1\}$ --- in analogy to the notation in \S\ref{sec:HC}-\S\ref{sec_decay}.

\noindent
\underline{Step 1 (preliminaries):} 
For $\bn \in \Z^2 \setminus\{0\}$, consider system~\eqref{lorentz_mode_n} with $\bC_\bn =\bR -\bJ_\bn$ satisfying~$\norm{\bR}=1$ and $\|\mJ_\bn\| = |\bn|$. 
By \eqref{est1} and Lemma \ref{lemmaiii} there exist constants $\kappa_1, \kappa_3>0$ such that
\[ \mR + \mJ_\bn \mR \mJ_\bn^* \ge \kappa_1 \mI,\quad \mR + \mC_\bn^* \mR \mC_\bn \ge \kappa_3 \mI\qquad \forall \bn \in \Z^2\setminus\{0\}. \]
Note that $\kappa_1,\kappa_3$ can be chosen independently of $\bn$ due to the identification of $\bJ_\bn$ with $|\bn|\bJ_{(1,0)}$, i.e.,
\[ \bR + |\bn|^2 \bJ_{(1,0)} \bR \bJ_{(1,0)}^* \ge \bR + \bJ_{(1,0)} \bR \bJ^*_{(1,0)} \ge \kappa_1 \bI\,, \]
and similarly for $\kappa_3$.

For $\|x\|=1$ we set, as in the proof of Proposition \ref{prop:short-t-decay:upper}, $\lambda_x \coleq \langle x, \bR x \rangle = \| \sqrt \bR x \|^2 \ge 0$ and $\mu^\bn_x \coleq \langle x, \bC_\bn^*\bR\bC_\bn x \rangle = \| \sqrt \bR \bC_\bn x \|^2 \ge 0$, which satisfy $\lambda_x + \mu^\bn_x \ge \kappa_3 > 0$.

Fixing $\delta \coleq \min(\kappa_1/5, \kappa_3/2)$ we have
\[ \mu^\bn_\delta \coleq \inf_{\substack{x \in \sphere_H\\ \lambda_x \le \delta}} \mu^\bn_x \ge \kappa_3 - \delta \ge \delta > 0\qquad \forall \bn \in \Z^2 \setminus\{0\}\,, \]
and analogously
\begin{equation}\label{c_eq2}
\inf_{\substack{x \in \sphere_H\\ \lambda_x \le \delta}} \| \sqrt \bR \bJ_\bn x \| \ge \sqrt{\kappa_1 - \delta} \ge 2 \sqrt{\delta} > 0 \qquad \forall \bn \in \Z^2 \setminus\{0\}.
\end{equation}

For $\lambda_x \le \delta$ we obtain
\[ \| \sqrt \bR \bJ_\bn x \| - \sqrt \delta \le \| \sqrt \bR \bC_\bn x \| \le \| \sqrt \bR \bJ_\bn x \| + \sqrt \delta \]
and, by taking the infimum,
\begin{equation}\label{c_eq4}
|\bn| \inf_{\substack{x \in \sphere_H\\\lambda_x \le \delta}} \| \sqrt \bR \bJ_{(1,0)} x \| - \sqrt \delta \le \sqrt{\mu^\bn_\delta} \le |\bn| \inf_{\substack{x \in \sphere_H\\\lambda_x \le \delta}} \| \sqrt \bR \bJ_{(1,0)} x \| + \sqrt{\delta}
\end{equation}
which, using~\eqref{c_eq2}, gives
\begin{equation} \label{mu_delta_n} 
 (2 |\bn| - 1 ) \sqrt \delta \le \sqrt{\mu^\bn_\delta} \le |\bn| + \sqrt \delta, 
\end{equation}
i.e., $\mu^\bn_\delta \sim |\bn|^2$.

The following two steps are refinements of the upper bound derived in Proposition~\ref{prop:short-t-decay:upper}:
We now set $g_\bn(x;t) \coleq \norm{\bP_\bn(t)x}^2 - 1$ for $\norm{x} = 1$.
To improve the estimate on~$g_\bn(x;t)$ in \eqref{est:g:C0:tilde} with $\mU_j =\mU^\bn_j$ and $\mC =\mC_\bn$, we use $\norm{\bR}=1$ and $\|\mJ_\bn\| = |\bn|$ to deduce $\|\mC_\bn\| \le 1 + |\bn| \le 2 |\bn|$ and, using~\eqref{eq1}, that
\begin{equation}
\label{c_eq3}
\| \mU^\bn_j\| \le 2 \| \mC_\bn\|^{j-1} 2^{j-1} \le 2 ( 4 |\bn|)^{j-1} \qquad\text{for } \bn \in \Z^2\setminus\{0\}.
\end{equation}

\noindent
\underline{Step 2 (estimate of $g_\bn(x;t)$ for $\lambda_x>\delta$):}

Using $\langle x, \mU^\bn_1 x \rangle=-2\lambda_x$ and \eqref{c_eq3} we have
\begin{equation}\label{c_eq5}
\begin{aligned}
g_\bn(x;t) &= - 2 \lambda_x t + \sum_{j=2}^\infty \frac{t^j}{j!} \langle x, \mU^\bn_j x \rangle \\
&\le - 2 \lambda_x t + \frac{e^{4|\bn|t}-1-4|\bn|t}{2|\bn|} \le -\lambda_x t \quad\text{for }|\bn|t \le \tau_1,
\end{aligned}
\end{equation}
where $\tau_1$ is chosen such that
\[ \delta = \delta_1(\tau_1) \coleq \frac{e^{4\tau_1} - 1 - 4 \tau_1}{2\tau_1} \]
and we have used $\delta < \lambda_x$; note that $\delta_1(0)=0$, $\delta_1$ is strictly monotone increasing and unbounded.

Next, for $t$ such that
\begin{align*}
    |\bn|t &\le \tau_2(\delta) \coleq \frac{\sqrt{12 \delta}}{\inf\limits_{\substack{x \in \sphere_H\\\sqrt{\lambda_x} \le \sqrt{\delta}}} \| \sqrt \bR \bJ_{(1,0)} x\| + \sqrt{\delta}} \\
    & \le \frac{\sqrt{12\delta}}{\inf\limits_{\substack{x \in \sphere_H\\\sqrt{\lambda_x} \le \sqrt{\delta}}}\|\sqrt \bR \bJ_{(1,0)} x \| + \frac{\sqrt{\delta}}{|\bn|}} \le \frac{\sqrt{12\delta}\,|\bn|}{\sqrt{\mu^\bn_\delta}}\,,
\end{align*}
where we used \eqref{c_eq4}, 
we obtain
\begin{equation}\label{c_eq6}
t \le \sqrt{\frac{12\delta}{\mu^\bn_\delta}} < \sqrt{\frac{12\lambda_x}{\mu^\bn_\delta}}
\end{equation}
for each $\bn$-mode separately. Combined with \eqref{c_eq5} this leads to
\begin{equation} \label{est_gn_lambda_greater_delta} 
 g_\bn(x;t) \le - \lambda_x t \le - \frac{\mu^\bn_\delta}{12}t^3 \qquad \text{for } x \in \sphere_H \text{ such that } \lambda_x>\delta \text{ and } |\bn|t \le \min(\tau_1, \tau_2). 
\end{equation}

\noindent
\underline{Step 3 (estimate of $g_\bn(x;t)$ for $\lambda_x\le\delta$):} 

From \eqref{est:g:C0:tilde} and $\|\bC_\bn\|\leq 2|\bn|$ we see that
\begin{align}
    \nonumber g(x;t) & \le \frac{1}{12} \langle \bC_\bn x, \mU^\bn_1 \bC_\bn x \rangle t^3 + \sum_{j=4}^\infty \frac{t^j}{j!} \sum_{k=1}^{j-2} \binom{j-1}k \Delta_{j,k}^{(1)} \langle (-\bC_\bn)^k x, \mU^\bn_1 (-\bC_\bn)^{j-k-1} x \rangle \\
    & \le - \frac{\mu^\bn_\delta}{6} t^3 + \sum_{j=4}^\infty t^j \frac{2^{j-1}}{j!} 2(2|\bn|)^{j-1} = - \frac{\mu^\bn_\delta}{6} t^3 + \frac{1}{2|\bn|} \sum_{j=4}^\infty \frac{(4 |\bn| t)^j}{j!} \label{c_eq7}\,,
\end{align}
where we have used the easily verified estimate
\[ \sum_{k=1}^{j-2} \binom{j-1}{k} \Delta^{(1)}_{j,k} \le 2^{j-1}. \]
Then, we use that
\begin{equation}\label{c_eq8}
\frac{1}{2|\bn|} \sum_{j=4}^\infty \frac{(4 |\bn| t)^j}{j!} = \frac{e^{4|\bn|t} - 1 - 4 |\bn|t - 8 |\bn|^2 t^2 - \frac{32}{3} |\bn|^3 t^3}{2|\bn|} \le \frac{\mu^\bn_\delta}{12}t^3 \qquad \text{for }|\bn|t \le \tau_3,
\end{equation}
where $\tau_3$ is chosen, using~\eqref{c_eq2} and~\eqref{c_eq4}, such that
\begin{multline*}
    \frac{e^{4|\bn|t} - 1 - 4 |\bn|t - 8 |\bn|^2 t^2 - \frac{32}{3} |\bn|^3 t^3}{2|\bn|^3 t^3} \le \frac{e^{4 \tau_3} - 1 - 4 \tau_3 - 8 \tau_3^2 - \frac{32}{3} \tau_3^3}{2 \tau_3^3} \\ \eqcol \delta_3(\tau_3) = \frac{\delta}{12} \le \frac{1}{12} \left( \inf_{\substack{x \in \sphere_H\\\lambda_x \le \delta}} \| \sqrt \bR \bJ_{(1,0)}\| - \frac{\sqrt \delta}{|\bn|}\right)^2 \le \frac{\mu_\delta^\bn}{12|\bn|^2}.
\end{multline*}
Together, \eqref{c_eq7} and \eqref{c_eq8} result in
\begin{equation} \label{est_gn_lambda_lesser_delta} 
 g_\bn(x;t) \le - \frac{\mu^\bn_\delta}{12}t^3 \qquad \text{for } x \in \sphere_H \text{ such that } \lambda_x\le\delta \text{ and } |\bn|t \le \tau_3. 
\end{equation}

In total, the estimates~\eqref{est_gn_lambda_greater_delta} and~\eqref{est_gn_lambda_lesser_delta} imply that, for arbitrary $x \in \sphere_H$,
\[ g_\bn(x;t) \le - \frac{\mu^\bn_\delta}{12}t^3 \quad \text{for }|\bn|t \le \tau \coleq \min(\tau_1, \tau_2, \tau_3, 1) \]
which, using~\eqref{mu_delta_n}, yields that 
\begin{equation}\label{c_eq9}
g_\bn(x;t) \le - \frac{(2|\bn|-1)^2\delta}{12}t^3 \le - \frac{|\bn|^2 \delta}{12} t^3 \quad \text{for }|\bn|t \le \tau.
\end{equation}

\noindent
\underline{Step 4 (coupling with exponential decay):} 

Following~\eqref{c_eq9}, our reference bound for $|\bn|=1$ is
\begin{equation} \label{reference_curve}  
 \| \bP_{(1,0)}(t) \| \le 1 - c_1 t^3 \quad\text{ on }[0,\tau] \quad\text{ with }c_1 \coleq \frac{\delta}{12} > 0. 
\end{equation}
To derive a uniform estimate for $\|\bP_\bn(t)\|=g_\bn(x;t)+1$ for $|\bn|> 1$, consider estimate~\eqref{c_eq9} for $0\leq t \leq \tau/|\bn|$: The endpoints $(\frac{\tau}{|\bn|}, 1- \frac{\delta \tau^3}{12|\bn|})$ of the final upper bound in~\eqref{c_eq9} are all below the reference curve $1-c_1 t^3$, see the cubic curves in Figure~\ref{fig:sketch_of_proof}. 

Next, due to~\eqref{c_eq9} and the semi-dissipativity of~\eqref{lorentz_mode_n}, the estimate
\begin{equation}\label{Pn-est}
  \|\bP_\bn(t)\| \leq 1-\frac{\delta \tau^3}{12|\bn|}
\end{equation}
holds for all $t\geq \tau/|\bn|$, but this horizontal line eventually crosses the cubic curve $1-c_1 t^3$, see Figure~\ref{fig:sketch_of_proof}.
Hence, a constant estimate until $t=\tau$ will not suffice.

Therefore, the estimate \eqref{Pn-est} is combined with the exponential decay estimate~\eqref{unif-decay} to deduce that
\begin{equation}\label{c_eq10}
\| \bP_\bn(t) \| \le \left(1-\frac{\delta \tau^3}{12|\bn|}\right) \min \left[1, \sqrt{1 + \frac{1}{|\bn|-1/2}}e^{-\lambda_0(t-\frac{\tau}{|\bn|})}\right],
 \qquad \text{for } t\geq \tau/|\bn|;
\end{equation}
where the factor $1$ applies on the interval $[\tau/|\bn|, t_\bn]$ with
\[ t_\bn = \frac{\tau}{|\bn|} + \frac{\ln \left( 1 + \frac{1}{|\bn|-1/2}\right)}{2\lambda_0} \le \frac{\tau}{|\bn|} + \frac{1}{2 \lambda_0} \frac{1}{|\bn|-1/2} \le \left(\tau + \frac{1}{\lambda_0}\right)\frac{1}{|\bn|} \eqcol \tilde t_\bn. \]
At this endpoint $t_\bn$, the estimate \eqref{Pn-est} still has to be below the reference curve, possibly with reduced $c$. Let
\[ c_2 \coleq \frac{c_1}{\left( 1 + \frac{1}{\lambda_0 \tau} \right)^3}. \]
Then, for $t\in[\tau/|\bn|,t_\bn]$, the estimate $\| \bP_\bn(t) \| \le 1 - c_2 t^3$ holds, 
due to~\eqref{c_eq10} and 
\[ 
 1-\frac{\delta \tau^3}{12|\bn|}
= 1 - c_1 \frac{\tau^3}{|\bn|}
\le 1 - c_2 \tilde{ t\mkern 0mu}_{\bn}^3 
\le 1 - c_2 t_\bn^3\qquad \forall \bn \in \Z^2 \setminus\{0\}. 
\] 
Again, the comparison of the reference curve in~\eqref{reference_curve} and the upper bound in~\eqref{c_eq10} is only done at the endpoint $t=t_\bn$ of the curve segments, since the curve segment of the reference curve is concave, and the other upper bound is a combination of a concave, a horizontal and a convex curve segment, respectively, see Figure~\ref{fig:sketch_of_proof}.

At the final endpoint $\tau$, all estimates still have to be below the new reference curve $1 - c_2 \tau^3$, possibly with $c$ reduced further: We only need to estimate \eqref{c_eq10} for such $|\bn|\ge r \ge 1$ ($r$ not necessarily integer) for which the exponential term in \eqref{c_eq10} is already strictly smaller than 1 at $t=\tau$ (or equivalently $t_\bn<\tau$); otherwise we are already done by the previous step.
So, we have
\begin{equation}\label{c_eq11} \| \bP_\bn(\tau) \| \le \left(1-\frac{\delta \tau^3}{12|\bn|}\right) \sqrt{ 1 + \frac{1}{|\bn| - 1/2}} e^{-\lambda_0 \frac{|\bn|-1}{|\bn|}\tau} \le 1 - c_3 \tau^3 \end{equation}
for the unique $c_3$ such that the second inequality in \eqref{c_eq11} is an equality with $|\bn| = r$.

Concluding, \eqref{Pn-decay} holds with $c \coleq \min(c_2, c_3)$.
\end{proof}
}


\end{document}

\begin{remark}
Special cases of the above theorem were pointed out to us by Laurent Miclo: 
In \S1 of \cite{MiMo13} the short time decay behavior of the Goldstein-Taylor model (a linear transport equation with relaxation term) was determined as $1-\frac{t^3}3+o(t^3)$. 
Actually, this model is a PDE. 
But since it is considered on a torus in $x$, each of its spatial Fourier modes (except of the 0-mode) satisfies a conservative-dissipative ODE system with hypocoercivity index 1 (see \cite{AAC16} for details of this modal decomposition). 
Hence, mode by mode, the result from \cite{MiMo13} is an example for Theorem~\ref{th:HC-decay}. 
For closely related BGK-models with hypocoercivity index 2 and 3 we refer to \cite{AAC18}.

In \cite{GaMi13} the short time decay behavior of a kinetic Fokker--Planck equation on the torus in $x$ was computed as $1-\frac{t^3}{12}+o(t^3)$. 
Again, in Fourier space and by using a Hermite function basis in velocity, this model can be written as an (infinite dimensional) conservative-dissipative system with hypocoercivity index 1 (see \S2.1 of \cite{GaMi13}). 
In that paper it was also mentioned that the decay exponent in~\eqref{short-t-decay} can be seen as some ``order of hypocoercivity'' of the generator.

For degenerate Fokker--Planck equations, the hypocoercivity index can also be related to the regularization rate for short times: 
In \cite[Theorem A.12]{Vi09} the regularization of initial data from a weighted $L^2$ space into a weighted $H^1$ space is derived, and in \cite[Theorem A.15]{Vi09}, \cite[Theorem 4.8]{ArEr14} it is generalized to entropy functionals and their corresponding Fisher informations. 
In all these cases the regularization rate is $t^{-a}$ with $a=2\mHC+1$ (somewhat related to Theorem~\ref{th:HC-decay} above).
\end{remark}